\documentclass[11pt]{amsproc}
\usepackage{amsmath,amsfonts,amssymb,amsthm}
\IfFileExists{amsrefs.sty}{%
  \usepackage[abbrev,lite,nobysame]{amsrefs}%
}{%
  \usepackage{cite}%
}
\usepackage{mathrsfs}
\usepackage{graphics,graphicx}
\usepackage[dvipsnames]{xcolor}
\usepackage[margin=1in]{geometry}
\usepackage{tikz}

\usepackage{enumitem}
\usepackage[colorlinks=true, pdfstartview=FitV, linkcolor=BrickRed,citecolor=black, urlcolor=black]{hyperref}

\usepackage[centercolon]{mathtools}
\makeatletter
\@ifpackageloaded{amsrefs}{}{}
\makeatother

\DeclareMathAlphabet{\mathup}{OT1}{\familydefault}{m}{n}
\newcommand{\dd}[1]{\mathop{}\!\mathup{d} #1}

\def\XXint#1#2#3{{\setbox0=\hbox{$#1{#2#3}{\int}$ }
\vcenter{\hbox{$#2#3$ }}\kern-.59\wd0}}

\renewcommand{\div}{\mathrm{div}\,}
\DeclareMathOperator{\curl}{curl}

\newcommand{\N}{\mathbb{N}}
\newcommand{\Z}{\mathbb{Z}}

\newcommand{\R}{\mathbb{R}}
\newcommand{\Co}{\mathbb{C}}
\newcommand{\T}{\mathbb{T}}

\DeclareMathOperator{\graph}{graph}
\newcommand{\Gr}{\mathrm{Gr}}

\newcommand{\xvec}{\boldsymbol{x}}
\newcommand{\yvec}{\boldsymbol{y}}
\newcommand{\zvec}{\boldsymbol{z}}
\newcommand{\vvec}{\boldsymbol{v}}

\newcommand{\cS}{\mathcal{S}}

\newcommand\e{{\rm e}}

\DeclareMathOperator{\Lip}{Lip}

\DeclareMathOperator*{\esssup}{ess\,sup}
\DeclareMathOperator*{\interior}{int\,}

\newcommand{\spann}{\mathrm{span}}

\newcommand{\eps}{\varepsilon}

\newcommand{\dist}{\text{\rm dist}}

\newtheorem{proposition}{Proposition}[section]
\newtheorem{theorem}{Theorem}[section]
\newtheorem{lemma}{Lemma}[section]
\newtheorem{corollary}{Corollary}[section]

\theoremstyle{definition}
\newtheorem{definition}{Definition}[section]

\theoremstyle{remark}
\newtheorem{remark}{Remark}[section]

\DeclareMathAlphabet{\mathup}{OT1}{\familydefault}{m}{n}

\numberwithin{equation}{section}

\makeatletter

\renewcommand\subsubsection{\@startsection{subsubsection}{3}%
\normalparindent{.5\linespacing\@plus.7\linespacing}{-.5em}
{\normalfont\bfseries}}

\def\@tocline#1#2#3#4#5#6#7{\relax
  \ifnum #1>\c@tocdepth 
  \else
    \par \addpenalty\@secpenalty\addvspace{#2}%
    \begingroup \hyphenpenalty\@M
    \@ifempty{#4}{%
      \@tempdima\csname r@tocindent\number#1\endcsname\relax
    }{%
      \@tempdima#4\relax
    }%
    \parindent\z@ \leftskip#3\relax \advance\leftskip\@tempdima\relax
    \rightskip\@pnumwidth plus4em \parfillskip-\@pnumwidth
    #5\leavevmode\hskip-\@tempdima
      \ifcase #1
       \or\or \hskip 1em \or \hskip 2em \else \hskip 3em \fi%
      #6\nobreak\relax
    \dotfill\hbox to\@pnumwidth{\@tocpagenum{#7}}\par
    \nobreak
    \endgroup
  \fi}
\makeatother

\begin{document}

\title[Exponential growth and decay in the ideal induction equation]{Exponential growth and decay in the ideal induction equation}

\author[V. Navarro-Fern\'andez]{V\'ictor Navarro-Fern\'andez}
\address{(VNF) Department of Mathematics, Imperial College London, London, UK}
\email{v.navarro-fernandez@imperial.ac.uk}

\date{\today}

\begin{abstract}
We construct a divergence-free velocity field on the three-dimensional torus that is time-periodic and consists of three alternating piecewise-affine shears. For sufficiently large shear amplitude, every non-zero divergence-free initial field in $L^p$ grows exponentially under the ideal induction equation. The proof establishes a uniform cone condition for the time-one map, and combines it with a bunching inequality to rule out nontrivial divergence-free fields lying almost everywhere in the stable bundle. Additionally, we show that for the time-reversed velocity, whose time-one map is the inverse of the original one, there exist nontrivial, bounded, divergence-free initial configurations taking values almost everywhere in its two-dimensional stable bundle. The corresponding solution decays exponentially in every $L^p$.
\end{abstract}

\maketitle

\setcounter{tocdepth}{1}
\tableofcontents

\section{Introduction}\label{s:intro}

Consider the ideal induction or vector transport equation
\begin{equation}\label{eq:dynamo}
\partial_t B+(u\cdot\nabla)B=(B\cdot\nabla)u, \quad \div B=0, \quad B(0,\cdot)=B_0,
\end{equation}
on \(\xvec\in\T^3=\R^3/(2\pi\Z)^3\). Here $u:\R_+\times\T^3\to\R^3$ denotes a prescribed divergence-free velocity field that transports and stretches the vector field $B:\R_+\times\T^3\to\R^3$ passively. $B_0:\T^3\to\R^3$ denotes an initial configuration for the Cauchy problem, that is taken divergence-free. In terms of the Lie bracket of vector fields, $[u,B]=(u\cdot\nabla)B-(B\cdot\nabla)u$, equation \eqref{eq:dynamo} is the Lie-transport equation
\[
\partial_tB + [u,B]=0.
\]
As such, we can recover solutions to \eqref{eq:dynamo} via push-forward with the volume-preserving flow map $\phi_t:\T^3\to\T^3$, associated to the divergence-free velocity field $u$, that is defined by the ODE
\[
\frac{\dd}{\dd t}\phi_t(\xvec)=u(t,\phi_t(\xvec)), \quad \phi_0(\xvec)=\xvec.
\]
Indeed, provided that we have sufficient regularity, the push-forward of the initial datum by the flow, given by $B(t,\cdot)=(\phi_t)_\ast B_0$, solves \eqref{eq:dynamo}. Since $B_0$ is a vector field, this can be rewritten into the solution formula
\begin{equation}\label{eq:solutions}
B(t,\xvec) =D\phi_t(\phi_t^{-1}(\xvec)) B_0(\phi_t^{-1}(\xvec)),
\end{equation}
which preserves the divergence constraint $\div B(t,\cdot)=0$ for all $t>0$. 
The same identity implies Alfv\'en's frozen-flux theorem. If a regular, oriented surface $\Sigma_0\subset\T^3$ is transported by the flow map to $\Sigma_t=\phi_t(\Sigma_0)$, then
\[
\int_{\Sigma_t}B(t,\xvec)\cdot n_{\Sigma_t}(\xvec) \dd{\mathscr H^2(\xvec)} = \int_{\Sigma_0}B_0(\yvec)\cdot n_{\Sigma_0}(\yvec) \dd{\mathscr H^2(\yvec)},
\]
where $n_\Sigma$ denotes a unit vector that is everywhere normal to the surface $\Sigma$, and $\mathscr{H}^2$ the two-dimensional Hausdorff measure. The geometric formulation of the induction equation and its relation to Lie advection has been extensively discussed in the literature, see e.g.\ \cites{GilbertVanneste21,AK98}.

The induction equation \eqref{eq:dynamo} can be understood as a particular case of the magnetohydrodynamics (MHD) equations. These are derived from Maxwell's equations together with Ohm's law, and yield the following (induction) equations for a magnetic field $B^\eps$,
\begin{equation}\label{eq:fast-dynamo}
\partial_t B^\eps + (u\cdot\nabla)B^\eps - (B^\eps\cdot \nabla)u = \eps\Delta B^\eps, \quad \div B^\eps = 0,
\end{equation}
where $\eps\geq 0$ denotes the magnetic resistivity.
In MHD, \eqref{eq:fast-dynamo} is typically coupled with the Navier--Stokes equations for the velocity $u$ of a fluid through a Lorentz force. They describe the evolution of a magnetic field $B^\eps$ transported by a conducting fluid flow that moves with velocity $u$. We can decouple the MHD equations, e.g.\ by neglecting the contribution of the Lorentz force, and look only at the induction equation \eqref{eq:fast-dynamo}. Despite being a simplified scenario, equations \eqref{eq:fast-dynamo} can provide useful information about the behaviour of solutions to the fully non-linear MHD equations, see \cites{NFV26}. For further information about the MHD equations we refer to \cites{Davidson01,Moffatt1978}.

On this note we are first interested in solutions for which there is a rapid creation of magnetic energy, represented by (the square of) the $L^2(\T^3)$ norm of $B^\eps(t,\cdot)$. This is known as the \emph{dynamo problem}, and it has attracted the attention of applied mathematicians and theoretical physicists since very early in the 1900s, see Larmor's seminal work \cite{Larmor19}. Following Arnold's definition \cites{ArnoldsProblems,AK98}, we say that a velocity field $u$ is a dynamo if, for any fixed $\eps\geq 0$, there exists an initial configuration $B_0^\eps\in L^2(\T^3)$ with $\div B_0^\eps=0$, starting from which, the $\eps$-dependent solution to the linear problem \eqref{eq:fast-dynamo} satisfies
\[
\gamma_\eps := \liminf_{t\to\infty} \frac{1}{t}\log\|B^\eps(t,\cdot)\|_{L^2}>0.
\]
Because of their connection to the dynamo problem, equations \eqref{eq:dynamo} and \eqref{eq:fast-dynamo} are also referred to as the \emph{kinematic dynamo} equations. For further information about the dynamo problem we refer to the monographs \cites{ChildressGilbert,AK98} and the doctoral thesis \cite{david-thesis}.

Finding dynamos is generally a difficult task. On the one hand, energy estimates in \eqref{eq:fast-dynamo} yield 
\[
\|B^\eps(t,\cdot)\|_{L^2} \leq \|B_0^\eps\|_{L^2}\e^{\int_0^t\|\nabla u(s,\cdot)\|_{L^\infty}\dd s} \quad\text{for all }t\geq 0.
\]
If $\nabla u$ is uniformly bounded, exponential is the fastest growth rate we can hope for. On the other hand, there exists a collection of \emph{antidynamo theorems} that rule out some velocity fields with symmetric geometries from being dynamos. These results showcase the necessity of certain complexity for the flows if one hopes for exponential growth, see e.g.\ \cites{zeldovich1980magnetic,Cowling33,Vishik89}. 

Dynamos can be classified by how $\gamma_\eps>0$ depends on $\eps$. If $\gamma_\eps>0$ for all $\eps>0$ sufficiently small and $\liminf_{\eps\to 0}\gamma_\eps>0$, we have a \emph{fast dynamo}. Examples of fast dynamos in unbounded domains are provided in \cites{gilbert1988,CZSV25,ZRMS1984}. In $\T^3$ an example has only been found very recently with Lipschitz velocity fields \cite{CZVS26+}, making crucial progress to solve a long-standing problem proposed by Arnold \cite{ArnoldsProblems}. As the current paper was being finished, more progress towards the resolution of Arnold's conjecture was made as an example of a Lipschitz autonomous field in $\T^3$ has also been discovered in \cite{Niebel6+}. Other relevant examples of \emph{subsequential} fast dynamos, for which only $\limsup_{\eps\to 0}\gamma_\eps>0$, have been recently discovered in $\T^3$ \cites{Rowan25,SorellaVillringer25+}, and in general bounded domains with boundary \cite{DNFLM26+}. If $\gamma_\eps>0$ for all $\eps>0$ sufficiently small but $\gamma_\eps\to 0$ as $\eps\to 0$, we have a \emph{slow dynamo}. Examples of slow dynamos in different subdomains of $\R^3$, including bounded domains, are provided in \cite{NFVillringer}. Finally, in the zero-resistivity case for which $\eps=0$, we say that we have an \emph{ideal dynamo} if $\gamma_0>0$. 

The latter is the type of dynamo action we will be putting a focus on this paper, thus it suffices to study the ideal induction equation \eqref{eq:dynamo}, whose solution we denote $B$. At first glance, the solution formula \eqref{eq:solutions} already hints a way to find ideal dynamos. Indeed, take $D\phi_t$ with an exponentially expanding/contracting pair of eigenvalues, i.e.\ a hyperbolic point. As shown in Appendix \ref{s:hyperbolic-point}, one can find an initial datum $B_0$ concentrated around the hyperbolic point so that $\|B(t,\cdot)\|_{L^p}$ grows exponentially fast for any $p>1$. This mechanism however does not work if we want to obtain exponential growth for \emph{every} divergence-free initial configuration in $L^2$.

\begin{definition}\label{def:universal-dynamo}
We call $u$ a \emph{universal ideal dynamo} if for every non-zero, divergence-free initial configuration $B_0\in L^2(\T^3)$ there exist constants $\lambda,C>0$ such that the solution $B$ to the induction equation \eqref{eq:dynamo} satisfies
\[
\|B(t,\cdot)\|_{L^2}\geq C\e^{\lambda t} \quad\text{for all }t\geq 0.
\]
\end{definition}

An autonomous velocity $u$ cannot define a universal ideal dynamo. Indeed let $B_0=u$, then since $[u,u]=0$, it follows that $B(t,\cdot) = u$ for all $t\geq 0$. There are nonetheless two relevant approaches to the universal ideal dynamo problem in the literature, both constructed with stochastic flows. First, \cite{BaxendaleRozovskii93} provides an annealed example given by a vector field in $\R^d$ that is smooth-in-space and white-in-time, for which there is exponential growth in any $L^p$ norm after averaging over the random velocity. More recently, in \cite{CotiZelatiNavarroFernandez} an example of an almost-sure universal ideal dynamo is constructed in $\T^3$ using a randomized version of the ABC flows.

In \cite{CotiZelatiNavarroFernandez}, the universal dynamo result is obtained together with a result about exponential mixing for all initial passive scalars in $\T^3$. By mixing we refer to the convergence $\|\rho(t,\cdot)\|_{H^{-1}}\to 0$ as $t\to\infty$ for mean-free solutions to the scalar transport equation 
\begin{equation}\label{eq:transport}
\partial_t\rho + u\cdot\nabla\rho = 0.
\end{equation}
The connection between mixing for passive scalars and dynamo for passive vectors can be best understood by looking at the problem in $\T^2$. In two-dimensions, any mean-free divergence-free vector field can be written as $B=\nabla^\perp\rho$ for some scalar function $\rho$, where $\nabla^\perp = (-\partial_2,\partial_1)$. A direct computation shows that if $B=\nabla^\perp\rho$ solves the induction equation \eqref{eq:dynamo} starting from $B_0=\nabla^\perp\rho_0$, then $\rho$ solves the transport equation \eqref{eq:transport} starting from $\rho_0$. Assume that $u$ is a (universal) exponential mixer, namely there are $C,\lambda>0$ such that $\|\rho(t,\cdot)\|_{H^{-1}} \leq C \e^{-\lambda t}\|\rho_0\|_{H^1}$ for all initial data $\rho_0\in H^1(\T^2)$ and mean-free. Since \eqref{eq:transport} conserves the $L^2$ norm, via interpolation we can write
\[
\|\rho_0\|_{L^2}^2 =\|\rho(t,\cdot)\|_{L^2}^2 \leq \|\rho(t,\cdot)\|_{H^1}\|\rho(t,\cdot)\|_{H^{-1}} \leq C\e^{-\lambda t}\|\rho_0\|_{H^1}\|B(t,\cdot)\|_{L^2}.
\]
Equivalently, if $u$ is a universal mixer for passive scalars, then the solution for the passive vector equation \eqref{eq:dynamo} satisfies
\[
\|B(t,\cdot)\|_{L^2} \geq \frac{\e^{\lambda t}}{C}\frac{\|B_0\|_{H^{-1}}^2}{\|B_0\|_{L^2}},
\]
for all non-zero, mean-free, divergence-free initial data $B_0\in L^2(\T^2)$. This shows that all universal mixers in $\T^2$ are also universal ideal dynamos in the class of mean-free initial data. As such, there are plenty of examples available in the literature, either on a deterministic setting \cites{ELM25,ElgindiZlatos19,ACM19}, or with more regular but random velocity fields \cites{NFSeis26,BBPS22,BCZG23,CoopermanRowan26}.

This argument also serves to show that the problem in $\T^3$ is more challenging and interesting because \eqref{eq:dynamo} does not seem to have such a straightforward relation with the equation for the passive scalar \eqref{eq:transport}. To the best of the author's knowledge, the relation between mixing for passive scalars and dynamo in three-dimensions remains to be fully understood.

\subsection{The velocity field and main results}

In this paper we construct a velocity field in $\T^3$ that is a universal ideal dynamo in the sense of Definition \ref{def:universal-dynamo}.
We also prove a complementary result for the time-reversed velocity: it admits a non-zero bounded divergence-free initial field whose solution decays exponentially in every \(L^p\).
Let \(h:\T\to[0,\pi]\) be the periodic tent function defined, using the representative in \([0,2\pi)\), by
\[
h(s)=|s-\pi|.
\]
We extend it periodically modulo $2\pi$ as per usual in the torus. Trivially, $h(s)$ is Lipschitz and smooth away from \(\{0,\pi\}\). For \(t\in\R\), put \(\tau=t-\lfloor t\rfloor\in[0,1)\), and define the velocity field
\begin{equation}\label{eq:velocity}
u_\alpha(t,x,y,z)
=3\alpha
\begin{cases}
(h(y),0,0), & 0\leq\tau<1/3,\\
(0,h(z),0), & 1/3\leq\tau<2/3,\\
(0,0,h(x)), & 2/3\leq\tau<1,
\end{cases}
\end{equation}
that depends on a parameter $\alpha\in\R$ to be specified later. Here we write $\xvec=(x,y,z)\in\T^3$. This vector field is a composition of shears, one-periodic in time, uniformly Lipschitz in space, and divergence-free. The factor \(3\) is introduced to compensate for the duration \(1/3\) of each pulse.

With this velocity field we can already state the first main result. We recall that throughout the paper, $B$ denotes the Lagrangian solution \eqref{eq:solutions} to the induction equation \eqref{eq:dynamo}, defined almost everywhere in $\T^3$, and that is the unique $L^p$ distributional solution.

\begin{theorem}\label{thm}
Let $1\leq p\leq\infty$, and let $B$ be the solution to the induction equation \eqref{eq:dynamo} starting from a non-zero divergence-free field $B_0\in L^p(\T^3)$ with the prescribed velocity field $u_\alpha$ from \eqref{eq:velocity}. There exists $\alpha_0>0$ such that for all $\alpha\geq \alpha_0$ the following claim holds true: 

There is a constant $\lambda_1>0$ (independent of $B_0$ and $p$) and a constant $\delta_p(B_0)>0$ such that
\[
\|B(t,\cdot)\|_{L^p} \geq \delta_p(B_0) \e^{\lambda_1 t} \quad \text{for all } t\geq 0.
\]
In particular, $u_\alpha$ is a universal ideal dynamo for any $\alpha\geq\alpha_0$.
\end{theorem}

\begin{remark}
From the analysis in the subsequent sections we get quantitative estimates for the constants $\lambda_1$ and $\delta_p(B_0)$. In particular from Proposition \ref{prop:conditional-growth}, Remark \ref{rmk:alpha-dep} and Corollary \ref{cor:iterated} we see that $\lambda_1=O(\log\alpha)$, whereas $\delta_p(B_0)$ is of the form
\[
\delta_p(B_0) = C_\alpha\|P^uB_0\|_{L^p},
\]
with $C_\alpha$ scaling like $\alpha^{-4}$. $P^u$ denotes the projection onto the unstable direction of a hyperbolic splitting given by the velocity field \eqref{eq:velocity}, see Section \ref{s:hyperbolic-splitting}. In Section \ref{s:L2-kernel} we show that $\|P^uB_0\|_{L^p}>0$ for all non-zero, divergence-free fields with $B_0\in L^p(\T^3)$, but we also show that $B_0$ can be chosen with $\|B_0\|_{L^p}=1$ and so that $\|P^uB_0\|_{L^p}$ is arbitrarily close to zero.
\end{remark}

The mechanism behind the universal growth result is strongly asymmetric under time reversal. For the forward time-one flow map \(\phi_{t=1}\), we prove in Section \ref{s:hyperbolic-splitting} that the stable bundle is one-dimensional, and Section \ref{s:L2-kernel} shows that it contains no non-zero divergence-free \(L^p\) field.  For \(\phi_{t=1}^{-1}\) the stable bundle is the original two-dimensional unstable bundle.  We prove that this bundle does contain a non-zero bounded divergence-free field, for which we obtain exponential decay of the magnetic energy. We call this construction ``time reversal" because for integer times there holds $\phi_n^{-1} = \phi_{-n}$, but this is not necessarily true for all real times.

Define the reversed velocity on one period by
\begin{equation}\label{eq:reversed-velocity}
\widetilde u_\alpha(t,\xvec) = -u_\alpha(1-t,\xvec), \quad \text{for } 0\leq t<1,
\end{equation}
and extend it one-periodically. For this velocity field we have the following result.

\begin{theorem}\label{thm:reversed-flow-decay}
Let \(\widetilde u_\alpha\) be the velocity field defined in \eqref{eq:reversed-velocity}. There exists $\alpha_0>0$ such that for any $\alpha\geq \alpha_0$ the following claim holds true:

There is a non-zero, divergence-free and mean-free initial datum $B_0\in L^\infty(\T^3)$ and constants $\lambda_2,C>0$, such that for any $1\leq p\leq \infty$, the solution $B$ to the induction equation \eqref{eq:dynamo} starting from $B_0$ and with velocity $\widetilde u_\alpha$, satisfies
\[
\|B(t,\cdot)\|_{L^p} \leq C\e^{-\lambda_2 t}\|B_0\|_{L^p} \quad \text{for all } t\geq 0.
\]
\end{theorem}

\begin{remark}
As it will become apparent in the proof, using the estimates from Corollary \ref{cor:iterated}, it follows that the coefficient $\lambda_2=O(\log\alpha)$, whereas $C=O(\alpha^4)$ as $\alpha$ grows larger.
\end{remark}

Decay under volume-preserving (Lie) transport for passive vectors has also been a relevant question in the context of the dynamo problem. Indeed, it is a conjecture by Arnold, see \cite{ArnoldsProblems}*{Problem 1991-1}, to find a volume-preserving diffeomorphism that makes the magnetic energy decrease to arbitrarily small values. This question is also known as the Sakharov--Zeldovich minimisation problem, for further information we refer to \cite{AK98}*{Chapter III.3}. The problem was originally proposed in a three-dimensional ball, coming from the physical interest in understanding stars. The question was addressed by Freedman \cite{Freedman99}, who constructed a family of volume-preserving diffeomorphisms \(\varphi_t\) of the three-dimensional ball for which the push-forward of a rotation field $\xi$ satisfies
\[
\|(\varphi_t)_*\xi\|_{L^2}^2=O(t^{-1}),
\]
while its \(L^\infty\) norm is unbounded. Freedman's model however cannot be realised as the iteration of one fixed autonomous or periodic flow. Exponential rates of decay are only known for hyperbolic constructions like the suspension of Arnold's cat-map, discussed as well in \cite{AK98}*{Chapters II.5 and V.4}. In this regard, Theorem \ref{thm:reversed-flow-decay} gives, to the author's knowledge, the first example of a periodic volume-preserving flow that makes the magnetic energy decrease exponentially fast in $\T^3$. Exponential rates of decay are optimal for velocity fields uniformly bounded in $W^{1,\infty}(\T^3)$, following the energy estimate for solutions to the ideal induction equation \eqref{eq:dynamo},
\[
\|B(t,\cdot)\|_{L^2} \geq \|B_0\|_{L^2}\e^{-\int_0^t \|\nabla u(s,\cdot)\|_{L^\infty}\dd s} \quad \text{for all }t\geq 0.
\]

The idea of considering this particular velocity $u_\alpha$ comes from two previous constructions.
A two-dimensional analogue of \eqref{eq:velocity} was first introduced by Elgindi, Liss and Mattingly in \cite{ELM25}, 
\[
u^{\mathrm{ELM}}_\alpha(t,x,y)
=2\alpha
\begin{cases}
(h(y),0), & 0\leq\tau<1/2,\\
(0,h(x)), & 1/2\leq\tau<1.
\end{cases}
\]
They prove that, for sufficiently large values of $\alpha$, this vector field mixes exponentially fast any mean-free passive scalar in $C^1$. As discussed before, in two-dimensions any universal mixer is a universal ideal dynamo, at least in the class of mean-free initial data. Although this argument does not carry over to $\T^3$, it gives a reasonable candidate for a universal ideal dynamo in \eqref{eq:velocity}. More recently, Liss and Mattingly in \cite{LissMattingly26+} used the same velocity field to prove a cumulative form of Batchelor's law for a passive scalar with deterministic forcing.  

In parallel, Coti Zelati, Sorella and Villringer extended $u_\alpha^{\mathrm{ELM}}$ to the three-torus in \cite{CZVS26+} by adding a third shear of the form $(0,0,g(x,y))$. Then they found a solution to \eqref{eq:fast-dynamo} that satisfies
\[
\|B(t,\cdot)\|_{L^2} \geq \e^{\gamma t}\|B_0\|_{L^2} \quad \text{for all } t\geq 0,
\]
where the constant $\gamma>0$ becomes independent of the magnetic resistivity $\eps>0$ as $\eps\to 0$, i.e.\ a fast dynamo. This is, to the author's knowledge, the first rigorous example of a Lipschitz fast dynamo in $\T^3$ available in the literature.

The choice of \eqref{eq:velocity} instead of the velocity constructed in \cite{CZVS26+} follows from the fact that we want to avoid having privileged directions. Indeed, we need a field that has some \emph{uniform hyperbolicity} property. Heuristically, this idea comes from the arguments in Appendix \ref{s:hyperbolic-point}: if $u$ has a hyperbolic point then it is a (non-universal) ideal dynamo. If we want universality, we need to make sure that all vectors end up seeing the expanding direction. 

We also remark that it remains open whether the velocity field \eqref{eq:velocity} is also an exponential mixer for passive scalars on \(\T^3\). Proving this would extend \cite{ELM25} to three dimensions. A major obstacle is that the Demers--Liverani theory of anisotropic Banach spaces \cite{DemersLiverani08}, which underlies the two-dimensional argument, does not directly apply in higher dimensions.

\emph{Outline:} The proof of Theorem \ref{thm} is divided in two main parts. First, we show that the velocity field \eqref{eq:velocity} is uniformly hyperbolic almost everywhere. We do so by proving that it satisfies a uniform cone condition in Section \ref{s:vectorfield}. Then, in Section \ref{s:hyperbolic-splitting} we translate the condition into a uniform measurable hyperbolic splitting of the tangent space almost everywhere. This already yields exponential growth for all initial data that are not fully supported on the stable directions of the hyperbolic splitting, as shown in Section \ref{s:conditional}. In Section \ref{s:L2-kernel} we rule out this possibility for divergence-free $L^p$ vector fields. The key ingredient for this last step is a bunching inequality that ensures that the Lipschitz norm of the flow map associated to \eqref{eq:velocity} is smaller than the strength of contraction along the stable directions of the hyperbolic splitting. In Section \ref{s:non-uniformity}, we show that although no divergence-free $L^p$ field is supported in the stable bundle, we can indeed construct vector fields with an arbitrarily small unstable component. Last, Section \ref{s:reversed-flow} is devoted to the proof of Theorem~\ref{thm:reversed-flow-decay}.  We construct flat local unstable discs for the piecewise-affine time-one map and use a two-dimensional stream-function construction on a positive-measure
family of discs to produce a non-zero bounded divergence-free field
tangent to the corresponding stable subspace.

\emph{Notation:} Along this note we write the non-subscripted norm $\|\cdot\|$ to denote the operator norm,
\[
\|L\| = \sup_{|v|=1}|Lv|,
\]
for either matrices and linear operators. If $A\subseteq \R^d$, we write $|A|$ for its $d$-dimensional Lebesgue measure, although if the context induces confusion we write $\mathscr L^d(A)$. We also write $a\lesssim b$ if there exists a constant $C>0$, independent of all the relevant quantities, such that $a\leq Cb$. Moreover we use the notation $a\asymp b$ if $b\lesssim a\lesssim b$.

\section{Uniform cone condition}\label{s:vectorfield}

In this section we study the hyperbolic properties of the velocity field \eqref{eq:velocity}. Recall that we define \(h:\T\to[0,\pi]\) to be the periodic tent function defined by $h(s)=|s-\pi|$, that is extended periodically modulo $2\pi$ in the torus. If $\phi_t$ denotes the flow map associated to the velocity field $u_\alpha$, the time-one map is given by the map $T=\phi_{t=1}$, defined by the composition
\[
T=T_3\circ T_2\circ T_1,
\]
where $T_1$, $T_2$ and $T_3$ are the following shears 
\begin{equation}\label{eq:shears}
\begin{split}
T_1(x,y,z)&=(x+\alpha h(y),y,z),\\
T_2(x,y,z)&=(x,y+\alpha h(z),z),\\
T_3(x,y,z)&=(x,y,z+\alpha h(x)).
\end{split}
\end{equation}
Space positions are always understood modulo \(2\pi\), and as before we write $\xvec=(x,y,z)\in\T^3$. Each shear is a volume-preserving bi-Lipschitz homeomorphism, hence so is \(T\).

The first property that we need this time-one flow map to have in order to prove Theorem \ref{thm} is a so-called uniform cone condition, that is defined as follows. 

\begin{definition}\label{def:cone}
Let \(T:\T^3\to\T^3\) be a bi-Lipschitz, piecewise \(C^1\) map with negligible singular set \(\cS\).  We say that \(T\) satisfies the \emph{uniform cone condition} if there are closed cones \(C_u,C_s\subset\R^3\) and constants \(C>0\), \(\gamma>1\) such that:
\begin{enumerate}
\item \(C_u=-C_u\), \(C_s=-C_s\), \(C_u\cap C_s=\{0\}\), and both cones have positive three-dimensional Lebesgue measure.  Moreover, there exist linear subspaces \(L_u\subset C_u\), \(L_s\subset C_s\) with
\[
\dim L_u+\dim L_s= \dim\R^3=3.
\]
\item For every \(\xvec\notin\cS\),
\[
DT(\xvec)(C_u\setminus\{0\})\subset\interior C_u, \quad [DT(\xvec)]^{-1}(C_s\setminus\{0\})\subset\interior C_s.
\]
\item Let $\mathcal{R}$ denote the regular set of $T$, i.e.\ the complement of the orbit of the singular set $\cS$,
\[
\mathcal R = \T^3\setminus\bigcup_{m\in\Z}T^m(\cS),
\]
which is a full measure set. Then for every \(\xvec\in\mathcal R\), \(n\geq1\), there holds
\[
|DT^n(\xvec)\vvec|\geq C\gamma^n|\vvec| \quad \forall\vvec\in C_u,
\]
\[
|DT^{-n}(\xvec)\vvec|\geq C\gamma^n|\vvec| \quad \forall\vvec\in C_s.
\]
\end{enumerate}
\end{definition}

This section is devoted to proving that the map \eqref{eq:shears} satisfies the uniform cone condition in $\T^3$, and this is precisely the claim of the following proposition.

\begin{proposition}\label{prop:uniform-cone}
There exists $\alpha_0>0$ such that for any $\alpha\geq\alpha_0$, the time-one map $T$ given by the composition of piecewise affine shears \eqref{eq:shears} satisfies the uniform cone condition from Definition \ref{def:cone} in $\T^3$, with $\dim L_u=2$ and $\dim L_s=1$.
\end{proposition}

This result will lead to uniform hyperbolicity as explained in Section \ref{s:hyperbolic-splitting}, and eventually to the proof of Theorem \ref{thm}. We remark that we choose positive values of $\alpha$, but by symmetry the same result holds true provided that $\alpha\leq \alpha_0$ for some threshold $\alpha_0<0$.

\subsection{Singular set and smoothness regions}
\label{sec:smoothness-regions}

First of all we observe that the map $T$ from \eqref{eq:shears} is Lipschitz, but its derivative is not well-defined everywhere on $\T^3$. There are certain points where $T$ becomes singular (the edge of the tents), and we need to introduce a notation that separates the singular set from the regions of smoothness where $DT$ is well-defined.

Let us denote by \(\Sigma=\{0,\pi\}\subset\mathbb T\) the singular set of \(h\). Then, the singular set of the time-one map \(T\) is
\[
\cS = \{y\in\Sigma\} \cup\{z\in\Sigma\} \cup\{x+\alpha h(y)\in\Sigma\},
\]
where as usual we assume all expressions to be mod $2\pi$.
It is a finite union of Lipschitz hypersurfaces, and therefore has zero Lebesgue measure.  Since \(T\) and \(T^{-1}\) are Lipschitz, the complement of the orbit of $\cS$
\[
\mathcal R = \T^3\setminus\bigcup_{n\in\Z}T^n(\cS)
\]
is a \(T\)-invariant set of full measure. We call $\mathcal{R}$ the (full-orbit) regular set of $T$.

\begin{figure}
\centering
\includegraphics[width=0.4\linewidth]{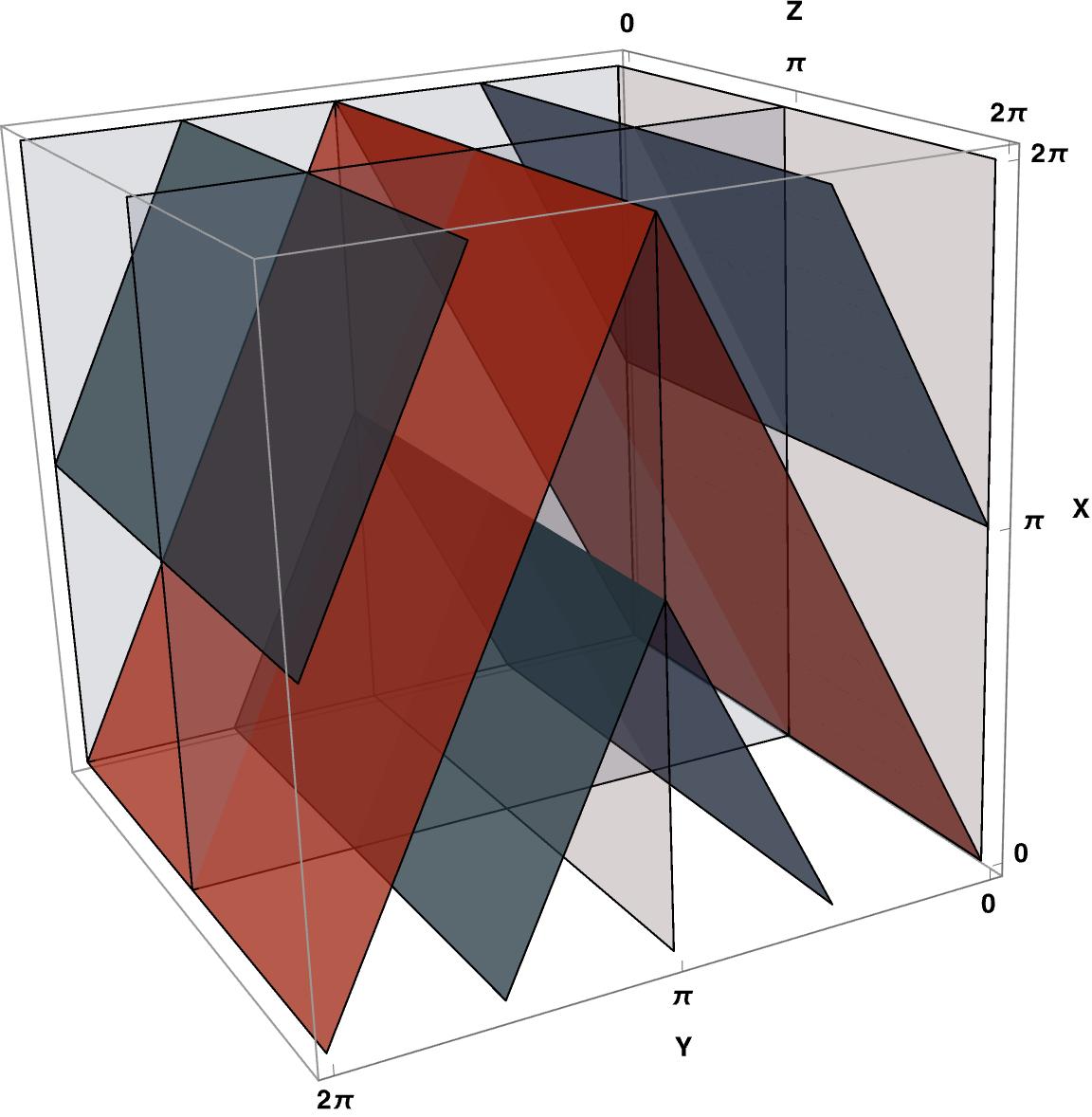}
\includegraphics[width=0.4\linewidth]{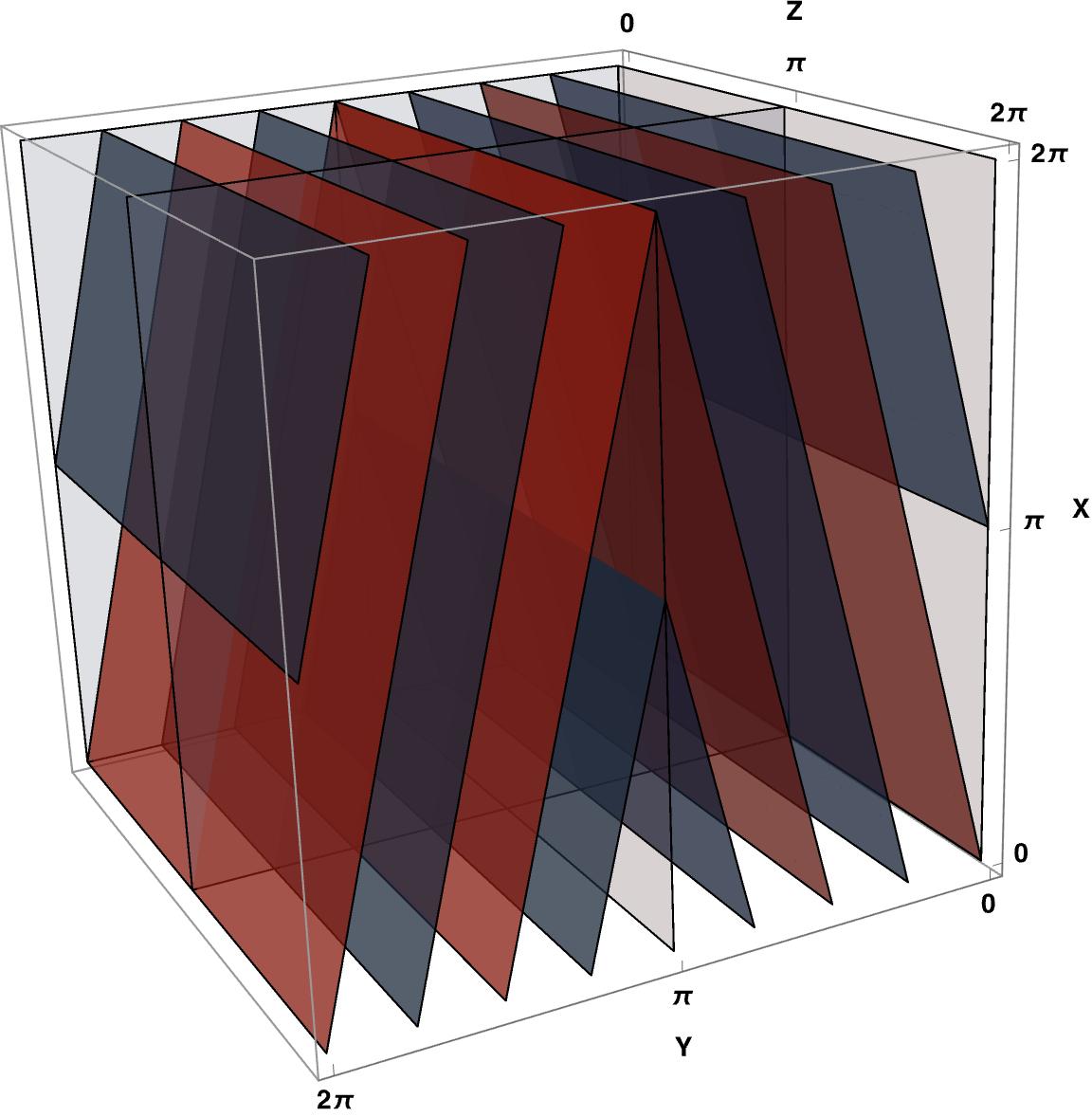}
\caption{Sections of the singular set \(\cS\) for \(\alpha=2\) (left)
and \(\alpha=4\) (right).}
\label{fig:sing-set}
\end{figure}

Let now $I_-=(0,\pi)$ and $I_+=(\pi,2\pi)$. For parameters \(\eps_1,\eps_2,\eps_3\in\{-1,+1\}\), we define
\[
\mathcal R_{\eps_1,\eps_2,\eps_3} = \left\{
(x,y,z): y\in I_{\eps_1}, z\in I_{\eps_2}, x+\alpha h(y)\in I_{\eps_3} \right\}.
\]
These are precisely the eight connected smoothness regions of the time-one map.  Indeed, the change of variables
\[
(x,y,z)\mapsto(x+\alpha h(y),y,z)
\]
maps each region homeomorphically onto \(I_{\eps_3}\times I_{\eps_1}\times I_{\eps_2}\). Graphic representation of the smoothness regions in $\T^3$ for different values of $\alpha$ can be observed in Figure \ref{fig:sing-set}.

Since the shears are affine, they have the property that the gradient matrix on each connected component \(\mathcal R_{\eps_1,\eps_2,\eps_3}\) is given by a matrix with real coefficients. More in detail, consider a point \(\xvec\in\mathcal R_{\eps_1,\eps_2,\eps_3}\) and define $a=\alpha\eps_1$, $b=\alpha\eps_2$, and $c=\alpha\eps_3$. Then the chain rule gives
\begin{equation}\label{eq:gradT}
DT(\xvec)=A_{a,b,c}
=
\begin{pmatrix}
1&a&0\\
0&1&b\\
c&ac&1
\end{pmatrix}.
\end{equation}
Thus all eight sign choices \(a,b,c\in\{-\alpha,\alpha\}\) occur, one on
each smoothness region.  Notice that \(\det A_{a,b,c}=1\) regardless of the choice of $a,b,c$ and thus $A_{a,b,c}\in SL_3(\R)$. Moreover we can explicitly compute the inverse matrix, that is given by
\begin{equation}\label{eq:inverse-DT}
A_{a,b,c}^{-1}
=
\begin{pmatrix}
1-abc&-a&ab\\
bc&1&-b\\
-c&0&1
\end{pmatrix}.
\end{equation}

\subsection{Uniform spectral estimates}

In this section we give spectral estimates for the gradient of the map $T$, defined in \eqref{eq:gradT}, that is piecewise constant in $\T^3$. All the arguments in this section follow simple computations and finite-dimensional linear algebra.

\begin{lemma}\label{lemma:eigenvalues}
There exists a constant \(\alpha_0>0\) such that, whenever \(\alpha\geq\alpha_0\), every matrix \(A_{a,b,c}\) has a distinguished real eigenvalue \(\lambda_s\) and two further eigenvalues \(\lambda_{u,1},\lambda_{u,2}\) satisfying
\begin{equation}\label{eq:eigenvalue-bounds}
|\lambda_s|\asymp  \alpha^{-3}\quad \text{and} \quad |\lambda_{u,j}|\asymp  \alpha^{3/2}.
\end{equation}
In particular, \(\lambda_s\) is the unique eigenvalue in the closed unit disc in \(\Co\) for large \(\alpha\).
\end{lemma}

\begin{proof}
Put \(K=abc=s\alpha^3\) with \(s\in\{-1,+1\}\).  A direct calculation gives
\[
\det(\lambda I-A_{a,b,c})=(\lambda-1)^3-K\lambda.
\]
Writing \(R=\alpha^3\) and \(\lambda=-st/R\), the characteristic equation becomes
\[
\left(1+\frac{st}{R}\right)^3-t=0.
\]
For sufficiently large \(R\), the left-hand side is positive at \(t=1/2\), negative at \(t=2\), and strictly decreasing on \([1/2,2]\).  Hence there is a (real) root
\[
\lambda_s=-\frac{s t_s}{R}, \quad \frac12<t_s<2,
\]
which proves the stable estimate. Following Vieta's formulas, the other two roots satisfy
\[
\lambda_{u,1}+\lambda_{u,2}=3-\lambda_s, \quad
\lambda_{u,1}\lambda_{u,2}=\lambda_s^{-1}.
\] 
Using these relations, the characteristic polynomial can be rewritten like 
\[
(\lambda-\lambda_s)(\lambda-\lambda_{u,1})(\lambda-\lambda_{u,2}) = (\lambda-\lambda_s)(\lambda^2-S\lambda+P)
\]
with $S = \lambda_{u,1}+\lambda_{u,2} = 3-\lambda_s$ and $P=\lambda_{u,1}\lambda_{u,2}=\lambda_s^{-1}$.
We hence compute the remaining roots,
\[
\lambda_{u,1},\lambda_{u,2}
=\frac12 (S\pm\sqrt{S^2-4P}).
\]
Here \(|S|\leq C\), and by the estimate for the stable eigenvalue, 
\(|P|\asymp\alpha^3\), therefore the modulus of \((S^2-4P)^{1/2}\) scales like $\alpha^{3/2}$, which is much larger that $|S|$ for large values of $\alpha$. This proves the remaining estimates in \eqref{eq:eigenvalue-bounds}. Moreover, Vieta's formulas give that when \(K>0\) the two unstable roots are real and have opposite signs, whereas when \(K<0\) they form a complex-conjugate pair.
\end{proof}

Given an eigenvalue \(\lambda\) and setting \(\mu=\lambda-1\), convenient right and left eigenvectors for the matrices \eqref{eq:gradT} are
\begin{equation}\label{eq:eigenvectors}
w_{\mathrm{R}}(\lambda)=\left(1,\frac{\mu}{a},\frac{\mu^2}{ab}\right), \quad w_{\mathrm{L}}(\lambda) =\left(1,\frac{\mu^2}{bc},\frac{\mu}{c}\right).
\end{equation}
Indeed, \(A_{a,b,c}w_{\mathrm{R}}(\lambda)=\lambda w_{\mathrm{R}}(\lambda)\) and \(w_{\mathrm{L}}(\lambda)^\top A_{a,b,c} =\lambda w_{\mathrm{L}}(\lambda)^\top\), or equivalently \(A_{a,b,c}^\top w_{\mathrm{L}}(\lambda)=\lambda w_{\mathrm{L}}(\lambda)\).
Moreover, the characteristic equation gives
\[
w_{\mathrm{L}}(\lambda)\cdot w_{\mathrm{R}}(\lambda)=1+\frac{2(\lambda-1)^3}{abc}=1+2\lambda.
\]
In particular, for the stable root there holds \(1+2\lambda_s\neq0\) when \(\alpha\geq\alpha_0\) is large, so we may set
\begin{equation}\label{eq:stable-normalization}
v^s_{a,b,c}=\frac{w_{\mathrm{R}}(\lambda_s)}{|w_{\mathrm{R}}(\lambda_s)|},\quad n^s_{a,b,c} = \frac{|w_{\mathrm{R}}(\lambda_s)|}{1+2\lambda_s}
w_{\mathrm{L}}(\lambda_s).
\end{equation}
Then \(|v^s_{a,b,c}|=1\) and \(n^s_{a,b,c}\cdot v^s_{a,b,c}=1\). The choice of \(v^s_{a,b,c}\) and \(n^s_{a,b,c}\) is uniform within smoothness regions, and so we make explicit that they depend on the choice of $a,b,c$. For each smoothness component $\mathcal{R}_{a,b,c}$, we define the subspaces
\[
E^s_{a,b,c}=\spann\{v^s_{a,b,c}\}, \quad E^u_{a,b,c}=\ker n^s_{a,b,c}.
\]
These are real invariant subspaces of dimensions one and two respectively. This definition characterises the unstable subspace only using the real stable eigenvalue. Finally, we introduce the local projector operators $P^s_{a,b,c}:\R^3\to E^s_{a,b,c}$ and $P^u_{a,b,c}:\R^3\to E^u_{a,b,c}$ via
\begin{equation}\label{eq:local-projectors}
P^s_{a,b,c}w=(n^s_{a,b,c}\cdot w)v^s_{a,b,c}, \quad P^u_{a,b,c}w=w-P^s_{a,b,c}w.
\end{equation}
The definition of the eigenvectors implies that both projectors commute with the matrix \(A_{a,b,c}\), indeed we compute
\[
P^s_{a,b,c}A_{a,b,c}w =(A_{a,b,c}^\top n^s_{a,b,c}\cdot\vvec)v^s_{a,b,c} =\lambda_s(n^s_{a,b,c}\cdot w)v^s_{a,b,c} =A_{a,b,c}P^s_{a,b,c} w,
\]
and so it follows for $P^u_{a,b,c}$ as well.

\begin{remark}\label{rmk:lipT}
This lemma shows that the spectral radius of $A_{a,b,c}$ scales like $\rho(A_{a,b,c})\asymp\alpha^{3/2}$, whereas one can see that the largest singular value scales like $\sigma_{\max}(A_{a,b,c}) = \|A_{a,b,c}\| \asymp \alpha^2$. For large values of $\alpha$ it holds $\sigma_{\max}(A_{a,b,c})>\rho(A_{a,b,c})$, which is due to the fact that $A_{a,b,c}$ is a highly non-normal matrix. In particular, this implies that the Lipschitz norm of $T$ must scale like the largest singular value, namely
\[
\Lip(T) = \esssup_{\xvec\in\T^3}\|DT(\xvec)\| \asymp \alpha^2.
\]
This will be relevant later in Section \ref{s:L2-kernel}.
\end{remark}

\subsection{The cone condition and proof of Proposition \ref{prop:uniform-cone}}\label{s:cones}

Let us label all eight matrices $A_{a,b,c}$ in \eqref{eq:gradT} by \(A_1,\ldots,A_8\), and write \(\lambda_{s,i}\), \(E_i^{s,u}\), \(P_i^{s,u}\), \(v_i^s\), and \(n_i^s\) for their associated stable eigenvalues, stable/unstable subspaces and projectors, and stable left/right eigenvectors. In this section we will prove that the piecewise affine vector field constructed in \eqref{eq:velocity} satisfies the uniform cone condition from Definition \ref{def:cone}, and thus we conclude with the proof of Proposition \ref{prop:uniform-cone}.

\begin{lemma}\label{lemma:geometry-cones}
Let $e_1=(1,0,0)$. Then, for any $\alpha$ large enough and any $1\leq i\leq 8$ we have
\begin{equation}\label{eq:projector-geometry}
|v_i^s-e_1|+|n_i^s-e_1|\lesssim \alpha^{-1}, \quad \|P_i^s\|+\|P_i^u\|\lesssim 1.
\end{equation}
Moreover, for every \(1\leq i,j\leq 8\), there holds
\begin{equation}\label{eq:projector-proximity}
\|P_i^s-P_j^s\|+\|P_i^u-P_j^u\| \lesssim \alpha^{-1}.
\end{equation}
\end{lemma}

\begin{proof}
Since \(|\lambda_s|=O(\alpha^{-3})\), formulas \eqref{eq:eigenvectors}--\eqref{eq:stable-normalization} readily give the first estimate in \eqref{eq:projector-geometry} by using the definition \eqref{eq:stable-normalization}.  From \eqref{eq:local-projectors} we get the identity \(P_i^s=v_i^s(n_i^s)^\top\), that gives the uniform operator bound in \eqref{eq:projector-geometry} and
\[
P_i^s-P_j^s =(v_i^s-v_j^s)(n_i^s)^\top +v_j^s(n_i^s-n_j^s)^\top.
\]
This, together with the definition \(P_i^u=I-P_i^s\), proves \eqref{eq:projector-proximity}.
\end{proof}

Now we fix \(\theta\in(0,1)\) independently of \(\alpha\), and define the local cones
\begin{equation}\label{eq:local-cones}
C_{s,i}=\{w\in\R^3:|P_i^uw|\leq\theta|P_i^sw|\}, \quad C_{u,i}=\{w\in\R^3:|P_i^sw|\leq\theta|P_i^uw|\}.
\end{equation}

\begin{figure}
    \centering
    \includegraphics[width=0.4\linewidth]{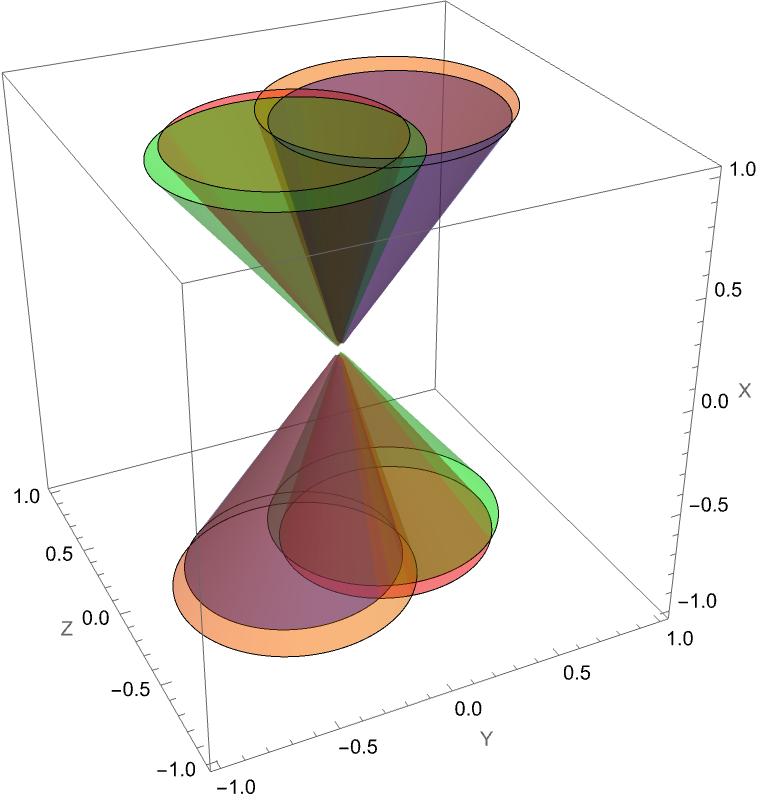}
    \includegraphics[width=0.4\linewidth]{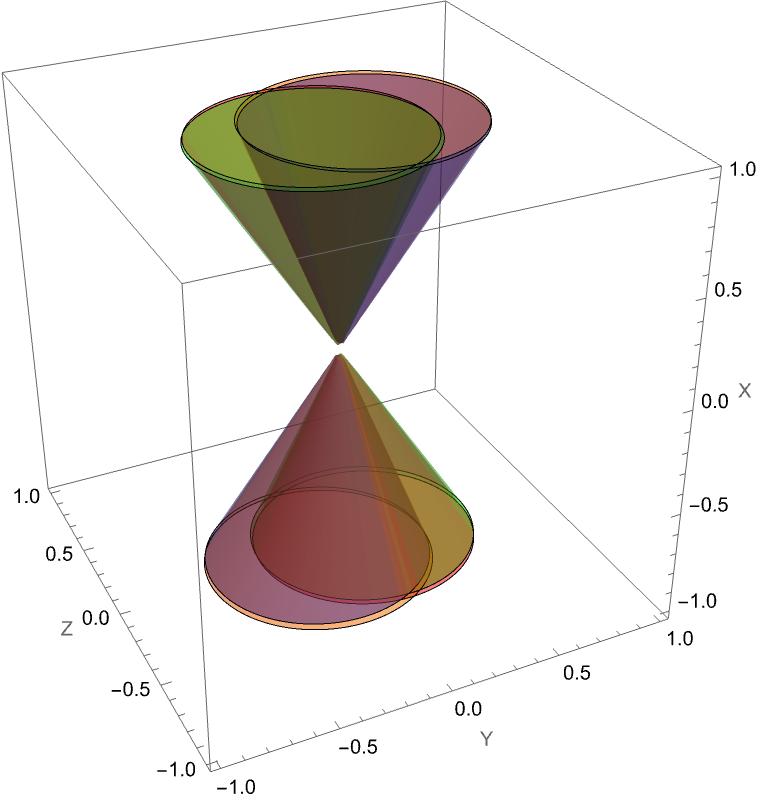}
    \caption{Cartoon of the stable cones for matrices of the form \eqref{eq:gradT} corresponding to $a=b=\alpha$ (red), $a=b=-\alpha$ (green), $a=-b=\alpha$ (blue), $a=-b=-\alpha$ (orange), and always $a=c$. We plotted $\theta = 1/2$ and $\alpha = 4$ (left), $\alpha=8$ (right).}
    \label{fig:stable-cones}
\end{figure}

In Figure \ref{fig:stable-cones} we can appreciate how as $\alpha>0$ grows bigger, all the local stable cones concentrate near the direction $e_1=(1,0,0)$ as claimed in Lemma \ref{lemma:geometry-cones}.

The following lemma gives useful estimates on the size of the opening for these local cones.

\begin{lemma}\label{lemma:opening}
Let $w\in\R^3$, then for any $1\leq i\leq 8$ there holds
\begin{equation}\label{eq:opening-u}
\frac{|P_i^sA_iw|}{|P_i^uA_iw|} \lesssim \alpha^{-4} \frac{|P_i^sw|}{|P_i^uw|}
\end{equation}
whenever \(P_i^uw\neq 0\), and
\begin{equation}\label{eq:opening-s}
\frac{|P_i^uA_i^{-1}w|}{|P_i^sA_i^{-1}w|} \lesssim \alpha^{-4} \frac{|P_i^uw|}{|P_i^sw|}
\end{equation}
whenever \(P_i^sw\neq 0\).
\end{lemma}

\begin{proof}
Since $E^u_i$ is invariant under $A_i$, the restriction $A_i|_{E_i^u}:E_i^u\to E^u_i$ is a well-defined linear endomorphism of the two-dimensional unstable subspace. Its determinant has modulus
\[
|\det (A_i|_{E_i^u})|=|\lambda_{u,1}\lambda_{u,2}|=|\lambda_s|^{-1}\gtrsim \alpha^3
\]
according to Lemma \ref{lemma:eigenvalues}, whereas by definition in \eqref{eq:gradT}, there holds \(\|A_i|_{E_i^u}\|\leq\|A_i\|\lesssim \alpha^2\). The two singular values of this two-dimensional restriction $0<\sigma_{\min}(A_i|_{E^u_i})\leq \sigma_{\max}(A_i|_{E^u_i})$ are given by
\[
\sigma_{\max}(A_i|_{E^u_i}) = \|A_i|_{E^u_i}\|, \quad \sigma_{\min}(A_i|_{E^u_i}) = \inf_{w\in E_i^u,\ |w|=1}|A_iw|.
\]
Since $\sigma_{\min}(A_i|_{E^u_i})\sigma_{\max}(A_i|_{E^u_i}) = |\det(A_i|_{E^u_i})|\gtrsim \alpha^3$ and $\sigma_{\max}(A_i|_{E^u_i})= \|A_i|_{E^u_i}\|\lesssim \alpha^2$, this yields 
\begin{equation}\label{eq:unstable-min-singular}
\inf_{w\in E_i^u,\ |w|=1}|A_iw| = \sigma_{\min}(A_i|_{E^u_i}) = \frac{|\det(A_i|_{E^u_i})|}{\sigma_{\max}(A_i|_{E^u_i})} \gtrsim \alpha.
\end{equation}
As seen before, the projectors commute with \(A_i\), so together with \eqref{eq:unstable-min-singular} we find
\[
|P_i^sA_iw|=|\lambda_s||P_i^sw| \lesssim \alpha^{-3}|P_i^sw|, \quad |P_i^uA_iw|=|(A_i|_{E^u_i})P_i^uw| \gtrsim \alpha|P_i^uw|
\]
and \eqref{eq:opening-u} follows. Similarly, using that
\begin{equation}\label{eq:unstable-min-singular-2}
\|(A_i|_{E^u_i})^{-1}\| = \frac{1}{\sigma_{\min}(A_i|_{E^u_i})} \lesssim \alpha^{-1}
\end{equation}
and the commutation with the inverse $A_i^{-1}$, we find
\[
|P_i^sA_i^{-1}w|=|\lambda_s^{-1}||P_i^sw| \gtrsim \alpha^3|P_i^sw|, \quad |P_i^uA_i^{-1}w|=|(A_i|_{E^u_i})^{-1}P_i^uw| \lesssim \alpha^{-1}|P_i^uw|,
\]
proving \eqref{eq:opening-s}.
\end{proof}

Next in order we use Lemma \ref{lemma:geometry-cones} to show that if one of the eight matrices $A_i$ has very narrow cones, then the same must be true for all the other matrices.

\begin{lemma}\label{lem:proximity}
Let $w\in\R^3$ and \(\alpha\) be sufficiently large, then the following holds true: 
\begin{enumerate}
    \item If \(|P_i^s w|\leq\eps|P_i^uw|\) for some $1\leq i\leq 8$ and $0<\eps<1$, then 
    \[
    |P_j^sw| \lesssim (\eps+\alpha^{-1})|P_j^uw|,
    \]
    for all $1\leq j \leq 8$.
    \item If \(|P_i^u w|\leq\eps|P_i^sw|\) for some $1\leq i\leq 8$ and $0<\eps<1$, then 
    \[
    |P_j^uw| \lesssim (\eps+\alpha^{-1})|P_j^sw|,
    \]
    for all $1\leq j\leq 8$.
\end{enumerate}
\end{lemma}

\begin{proof}
We prove \emph{(1)} first. By assumption, \(|w| \leq (1+\eps)|P_i^uw|\) so using Lemma \ref{lemma:geometry-cones} we write, on the one hand
\[
|P^s_jw| \leq |P^s_iw| + \|P^s_j-P^s_i\||w| \lesssim \eps|P^u_iw| + \alpha^{-1}(1+\eps)|P_i^uw|
\]
while on the other hand,
\[
|P^u_jw| \geq |P_i^uw| -\|P_j^u-P_i^u\||w| \gtrsim |P_i^uw|-\alpha^{-1}(1+\eps)|P_i^uw|.
\]
Putting both estimates together it yields \emph{(1)}. The proof of \emph{(2)} is identical with \(s\) and \(u\) interchanged.
\end{proof}

Finally we will combine all these intermediate lemmas to prove that there exist uniform cones that are common to all eight smoothness regions. Define the global cones
\begin{equation}\label{eq:global-cones}
C_u=\bigcap_{i=1}^8C_{u,i},\quad C_s=\bigcap_{i=1}^8C_{s,i},
\end{equation}
where we recall, the local cones $C_{u,i}$ and $C_{s,i}$ are defined in \eqref{eq:local-cones}. Then we have the following proposition.

\begin{proposition}\label{prop:cones}
There exists a constant \(\alpha_0>0\) such that, for all
\(\alpha\geq\alpha_0\) and all \(1\leq i\leq 8\), the global cones defined in \eqref{eq:global-cones} satisfy
\begin{enumerate}
    \item $A_i(C_u\setminus\{0\})\subset\interior C_u$ and $A_i^{-1}(C_s\setminus\{0\})\subset\interior C_s$;
    \item $|A_iw|\gtrsim \alpha|w|$ for all $w\in C_u$ and $|A_i^{-1}w|\gtrsim \alpha^3|w|$ for all $w\in C_s$;
    \item there exist linear subspaces $L_u=e_1^\perp\subset C_u$ and $L_s=\spann\{e_1\}\subset C_s$.
\end{enumerate}
\end{proposition}

We remark that we denote by $e_1^\perp$ the hyperplane of $\R^3$ that is perpendicular to $e_1$, meaning $e_1^\perp = \{v\in\R^3:e_1\cdot v = 0\}$, which is a two-dimensional subspace.

\begin{proof}
Let \(w\in C_u\setminus\{0\}\) so Lemma \ref{lemma:opening} and the definition of the cones \eqref{eq:local-cones} first gives that there exists a constant $C>0$ such that
\[
\frac{|P_i^sA_iw|}{|P_i^uA_iw|} \leq C\theta\alpha^{-4}=:\eps_\alpha.
\]
For large $\alpha$, $\eps_\alpha>0$ becomes small, so Lemma \ref{lem:proximity} applied with $\eps=\eps_\alpha$ gives for every $1\leq j\leq 8$,
\[
\frac{|P_j^sA_iw|}{|P_j^uA_iw|} \leq C(\alpha^{-4}+\alpha^{-1})<\theta
\]
provided that $\alpha$ is large. All eight inequalities are strict, proving the first inclusion in \emph{(1)}. The proof for \(A_i^{-1}C_s\) is the same, using \eqref{eq:opening-s} and again Lemma \ref{lem:proximity}.

Let now $w\in C_u$. The local splitting for $A_i$ and \eqref{eq:unstable-min-singular} give
\[
|A_iw| \geq|A_iP_i^uw|-|\lambda_sP_i^sw| \geq (c\alpha-C\theta\alpha^{-3})|P_i^uw|
\gtrsim \alpha|w|.
\]
Here we used \(|w|\leq(1+\theta)|P_i^uw|\).  Similarly, if \(w\in C_s\), using \(|w|\leq(1+\theta)|P_i^sw|\) we get 
\[
|A_i^{-1}w| \geq|\lambda_s|^{-1}|P_i^sw|
-\|(A_i^u)^{-1}\|\,|P_i^uw| \geq (c\alpha^3-C\theta\alpha^{-1})|P_i^sw|
\gtrsim \alpha^3|w|.
\]

Last we observe that if $w\in e_1^\perp = \{v\in\R^3:e_1\cdot v = 0\}$ then Lemma \ref{lemma:geometry-cones} gives
\[
|P_i^sw| = |n_i^s\cdot w| = |(n_i^s-e_1)\cdot w| \leq |n_i^s-e_1||w| \lesssim \alpha^{-1}|w|
\]
and $|P^u_iw|\gtrsim (1-\alpha^{-1})|w|$, thus for sufficiently large $\alpha$ there holds $e_1^\perp\subset C_u$. Likewise, we compute
\[
|P^u_ie_1| = |e_1-P_i^se_1| = |e_1-(n_i^s\cdot e_1)v_i^s| \leq |1-n_i^s\cdot e_1|+|n_i^s\cdot e_1||v_i^s-e_1|.
\]
Lemma \ref{lemma:geometry-cones} yields that the second term in the right-hand side is of order $\alpha^{-1}$, whereas the first term can be expanded
\[
|1-n_i^s\cdot e_1| = |(e_1-n_i^s)\cdot e_1| \leq |e_1-n_i^s| \lesssim \alpha^{-1},
\]
again due to Lemma \ref{lemma:geometry-cones}. All in all, $|P^u_ie_1|\lesssim \alpha^{-1}$ and thus $|P^s_ie_1|\gtrsim 1-\alpha^{-1}$, which gives \(\spann\{e_1\}\subset C_s\).  The strict versions of these inequalities also show that both global cones have nonempty interior and hence positive three-dimensional measure.
\end{proof}

This proposition has the following direct consequence.

\begin{corollary}\label{cor:iterated}
For any $\alpha$ sufficiently large, there exist \(\Lambda_u\asymp \alpha>1\) and \(\Lambda_s \asymp \alpha^3 >1\) such that for all \(\xvec\in\mathcal R\) and all \(n\geq1\),
\[
DT^n(\xvec)(C_u\setminus\{0\})\subset\interior C_u \quad \text{and} \quad |DT^n(\xvec)w|\geq\Lambda_u^n|w| \quad \forall w\in C_u,
\]
\[
DT^{-n}(\xvec)(C_s\setminus\{0\})\subset\interior C_s, \quad \text{and} \quad |DT^{-n}(\xvec)w|\geq\Lambda_s^n|w| \quad\forall w\in C_s.
\]
\end{corollary}

\begin{proof}
Apply Proposition~\ref{prop:cones} successively along the regular orbit.
\end{proof}

\begin{proof}[Proof of Proposition \ref{prop:uniform-cone}]
The map \(T\) is a volume-preserving bi-Lipschitz homeomorphism, and \(\cS\) is negligible. By definition, \(C_u\) and \(C_s\) are closed and symmetric. If \(w\in C_u\cap C_s\), then, for any fixed \(1\leq i\leq 8\),
\[
|P_i^sw|\leq\theta|P_i^uw| \leq\theta^2|P_i^sw|.
\]
Since \(\theta<1\), both projected components vanish, and hence \(w=0\). Proposition \ref{prop:cones}, Corollary \ref{cor:iterated}, and the subspaces $L_s,L_u$ from Proposition \ref{prop:cones}\emph{(3)} verify every item in Definition \ref{def:cone}.
\end{proof}

\section{Hyperbolic splitting and conditional growth}
\label{s:hyperbolic-splitting}

In this section we use the uniform cone condition from Section \ref{s:cones} to prove that the map $T$ provides a measurable uniformly hyperbolic splitting of the tangent space (i.e.\ $\R^3$) in the classical sense. This result is customary, typically covered under certain smoothness assumptions in the standard literature about hyperbolic systems, see e.g.\ \cite{KatokHasselblatt95}*{Section 6.4}. We nevertheless include here some details because $T$ is only piecewise $C^1$, the full-orbit regular set $\mathcal{R}$ is not compact, and we require estimates uniform in all eight smoothness regions.

Let us introduce the terminology required to make the arguments in this section rigorous. For \(k\in\{1,2\}\), the Grassmannian
\[
\mathrm{Gr}_k(\R^3)=\{F\subset\R^3:F\text{ is a linear subspace and }\dim F=k\}
\]
is the space of \(k\)-dimensional linear subspaces of \(\R^3\). Thus, \(\mathrm{Gr}_1(\R^3)\) is the space of unoriented lines through the origin, while \(\mathrm{Gr}_2(\R^3)\) is the space of planes through the origin.

We equip \(\mathrm{Gr}_k(\R^3)\) with the distance
\[
d_{\mathrm{Gr}}(F,G) = \|\Pi_F-\Pi_G\|,
\]
where \(\Pi_F\) denotes orthogonal projection $\Pi_F:\R^3\to F$. This distance measures the angle between the two subspaces. The Grassmannians are compact metric spaces.  In particular, every Cauchy sequence of \(k\)-dimensional subspaces converges to a
\(k\)-dimensional subspace. The orthogonal projections \(\Pi_F\) used here should be distinguished from the generally nonorthogonal spectral projections \(P_i^u,P_i^s\)
introduced in Section~\ref{s:cones}.

Every invertible linear map \(A:\R^3\to\R^3\) acts on the Grassmannians by
\[
F\mapsto AF=\{Av:v\in F\}.
\]
We define
\[
\mathscr G_u=\left\{F\in\mathrm{Gr}_2(\R^3):F\setminus\{0\}\subset C_u\right\},
\]
the family of two-dimensional subspaces entirely contained in the unstable cone, and
\[
\mathscr G_s=\left\{\ell\in\mathrm{Gr}_1(\R^3):\ell\setminus\{0\}\subset C_s\right\},
\]
the family of lines entirely contained in the stable cone. These families are nonempty because as shown in Proposition \ref{prop:cones},
\[
e_1^\perp\in\mathscr G_u, \quad \spann\{e_1\}\in\mathscr G_s.
\]
They are closed subsets of the corresponding Grassmannians and are therefore compact.

The first ingredient that we need to introduce is a graph representation formula. To understand this better let us fix one of the matrices \(A_i\) from one of the eight smoothness regions, and recall its invariant splitting
\[
\R^3=E_i^u\oplus E_i^s,
\]
which is a consequence of the computations in Lemma \ref{lemma:eigenvalues}. These subspaces are used only as coordinates, the orbit-dependent subspaces \(E^u(\xvec)\) and \(E^s(\xvec)\) will be constructed below.

If \(F\in\mathscr G_u\), then the restriction
\[
P_i^u|_F:F\to E_i^u
\]
is an isomorphism.  Indeed, if \(w\in F\) and \(P_i^uw=0\), then the local unstable-cone inequality gives
\[
|P_i^sw| \leq\theta|P_i^uw|=0,
\]
and hence \(w=0\). Since both \(F\) and \(E_i^u\) have dimension two, injectivity implies surjectivity. Consequently, there is a unique linear map
\[
L_F:E_i^u\to E_i^s,\quad L_F = P_i^s\circ(P_i^u|_F)^{-1},
\]
such that
\[
F = \graph(L_F) = \{v_u+L_Fv_u:v_u\in E_i^u\}.
\]
Moreover, since \(F\subset C_{u,i}\),
\[
\|L_F\|\leq\theta.
\]
Similarly, every \(\ell\in\mathscr G_s\) can be written uniquely as
\[
\ell = \graph(R_\ell) = \{v_s+R_\ell v_s:v_s\in E_i^s\},
\]
where
\[
R_\ell:E_i^s\to E_i^u, \quad \|R_\ell\|\leq\theta.
\]
We remark that on the set of graphs with slope bounded by $\theta$, convergence in operator norm of the graph maps is equivalent to convergence in $d_{\Gr}$. The constants are uniform because the splittings $E_i^u\oplus E_i^s$ are uniformly transverse.

\begin{lemma}\label{lem:graph-transform}
For all sufficiently large \(\alpha\), there exists \(q\in(0,1)\), with \(q\lesssim \alpha^{-4}\), for which every \(A_i\) of the form \eqref{eq:gradT} has the following property,
\[
A_i\mathscr G_u\subset\mathscr G_u, \quad A_i^{-1}\mathscr G_s\subset\mathscr G_s,
\]
and
\begin{equation}\label{eq:graph-transform-contraction}
\begin{split}
d_{\rm Gr}(A_iF,A_iG) & \leq q\,d_{\rm Gr}(F,G)
\quad \forall F,G\in\mathscr G_u,\\
d_{\rm Gr}(A_i^{-1}\ell,A_i^{-1}m) & \leq q\,d_{\rm Gr}(\ell,m) \quad \forall\ell,m\in\mathscr G_s.
\end{split}
\end{equation}
\end{lemma}

\begin{proof}
Fix \(i\) and use the invariant splitting \(\R^3=E_i^u\oplus E_i^s\). Every \(F\in\mathscr G_u\) is the graph of a map \(L_F:E_i^u\to E_i^s\) with \(\|L_F\|\leq\theta\). Its image is the graph of
\[
L_{A_iF}=\lambda_{s,i}L_F(A_i^u)^{-1},
\]
where $\lambda_{s,i}$ denotes the stable eigenvalue for the matrix $A_i$. By Lemma~\ref{lemma:eigenvalues} and \eqref{eq:unstable-min-singular-2},
\[
|\lambda_{s,i}|\,\|(A_i^u)^{-1}\|\lesssim \alpha^{-4}.
\]
On graphs of norm at most \(\theta\), graph norm and Grassmannian distance are uniformly equivalent. More precisely, there are uniform constants \(c,C>0\) such that
\[
c\|L_F-L_G\|
\leq d_{\rm Gr}(F,G)
\leq C\|L_F-L_G\|.
\]
The constants are uniform in \(i\) by Lemma~\ref{lemma:geometry-cones}.  Hence the contraction factor is of order
\[
\max_i(|\lambda_{s,i}|\,\|(A_i^u)^{-1}\|) \lesssim \alpha^{-4}.
\]
Every \(\ell\in\mathscr G_s\) is similarly the graph of a map \(R_\ell:E_i^s\to E_i^u\).  Under \(A_i^{-1}\), its graph transform is
\[
R_{A_i^{-1}\ell} =\lambda_{s,i}(A_i^u)^{-1}R_\ell,
\]
which gives the second estimate.
\end{proof}

Having introduced the Grassmannian and graph language, and with the contraction given by Lemma \ref{lem:graph-transform} in our belt, we can formulate the rigorous measurable uniformly hyperbolic splitting given by iterations of the time-one map $T$ from \eqref{eq:shears}.

Let \(\xvec\in\mathcal R\) and write
\[
A(\xvec)=DT(\xvec), \quad A^{(n)}(\xvec) = A(T^{n-1}\xvec)\cdots A(T\xvec)A(\xvec) = DT^n(\xvec)\footnote{Unless otherwise is specified, throughout this paper we write $DT^n(\xvec)=D(T^n)(\xvec)$ as defined here. To make a distinction with the $n$-th power of the matrix we write $DT(\xvec)^n$.}.
\]
We use the convention \(A^{(0)}=I\).

\begin{proposition}[Uniformly hyperbolic splitting]\label{prop:measurable-splitting}
There exist measurable maps
\[
\xvec\mapsto E^u(\xvec)\in{\rm Gr}_2(\R^3), \quad \xvec\mapsto E^s(\xvec)\in{\rm Gr}_1(\R^3)
\]
defined for all $\xvec\in\mathcal R$ such that
\begin{equation}\label{eq:bundle-direct-sum}
\R^3=E^u(\xvec)\oplus E^s(\xvec),
\end{equation}
\[
A(\xvec)E^u(\xvec)=E^u(T\xvec), \quad A(\xvec)E^s(\xvec)=E^s(T\xvec),
\]
\[
E^u(\xvec)\setminus\{0\}\subset \interior C_u, \quad E^s(\xvec)\setminus\{0\}\subset \interior C_s.
\]
Moreover, we have the expansion and contraction estimates
\begin{equation}\label{eq:bundle-expansion}
|A^{(n)}(\xvec)v_u| \geq\Lambda_u^n|v_u|, \quad \forall v_u\in E^u(\xvec), 
\end{equation}
\begin{equation}\label{eq:bundle-contraction}
|A^{(n)}(\xvec)v_s| \leq\Lambda_s^{-n}|v_s|, \quad \forall v_s\in E^s(\xvec),
\end{equation}
for all \(n\geq0\), and where the constants $\Lambda_u,\Lambda_s>1$ are defined in Corollary \ref{cor:iterated}. The splitting is uniformly transverse, i.e.\ there exists \(\kappa>0\) such
that
\begin{equation}\label{eq:uniform-transversality}
|v_u+v_s| \geq\kappa\max\{|v_u|,|v_s|\}
\end{equation}
for all \(\xvec\in\mathcal{R}\), \(v_u\in E^u(\xvec)\) and \(v_s\in E^s(\xvec)\).
\end{proposition}

\begin{proof}
Choose \(F_\star\in\mathscr G_u\) and \(\ell_\star\in\mathscr G_s\), for instance the subspaces from Proposition \ref{prop:cones}\emph{(3)}.  Set
\[
F_n(\xvec) = A^{(n)}(T^{-n}\xvec)F_\star \quad \text{and} \quad \ell_n(\xvec) = [A^{(n)}(\xvec)]^{-1}\ell_\star.
\]
Writing $\Gamma_{\xvec}(F) = A(\xvec)F$, we have 
\[
F_n(\xvec) = \Gamma_{T^{-1}\xvec}\circ \hdots \circ \Gamma_{T^{-n}\xvec}(F_\star),
\]
\[
F_{n+1}(\xvec) = \Gamma_{T^{-1}\xvec}\circ \hdots \circ \Gamma_{T^{-n}\xvec}(\Gamma_{T^{-(n+1)}\xvec}(F_\star)),
\]
for the $n$ and $n+1$ iterations respectively. By Lemma \ref{lem:graph-transform}, there exists $q\in (0,1)$ so that the common $n$-fold composition is $q^n$-Lipschitz, therefore
\[
d_{\rm Gr}(F_{n+1}(\xvec),F_n(\xvec)) \leq q^n\operatorname{diam}({\rm Gr}_2(\R^3)).
\]
For any $m>n$, the triangle inequality and the preceding estimate give
\[
d_{\mathrm{Gr}}(F_m(\xvec),F_n(\xvec)) \leq \sum_{k=n}^{m-1} d_{\mathrm{Gr}}(F_{k+1}(\xvec),F_k(\xvec)) \leq \operatorname{diam}(\mathrm{Gr}_2(\mathbb R^3)) \sum_{k=n}^{m-1}q^k \leq \operatorname{diam}(\mathrm{Gr}_2(\mathbb R^3)) \frac{q^n}{1-q}.
\]
The right-hand side is independent of $\xvec$ and tends to zero as $n\to\infty$. Thus $(F_n)_n$ is uniformly Cauchy. Since $\mathrm{Gr}_2(\mathbb R^3)$ is complete (in fact compact), $F_n$ converges uniformly in $\xvec$ to a map that we denote by $E^u(\xvec)$. The same argument gives, for every $m>n$,
\[
d_{\mathrm{Gr}}(\ell_m(\xvec),\ell_n(\xvec)) \leq \operatorname{diam}(\mathrm{Gr}_1(\mathbb R^3))\frac{q^n} {1-q},
\]
and hence $\ell_n$ converges uniformly in $\xvec$ to a map that we denote by $E^s(\xvec)$.
Finite products depend measurably on \(\xvec\), so the limits are measurable. The identities
\[
A(\xvec)F_n(\xvec)=F_{n+1}(T\xvec), \quad A(\xvec)\ell_{n+1}(\xvec)=\ell_n(T\xvec)
\]
give invariance after passing to the limit. The cone inclusions follow initially from closedness, i.e.\ $E^u(\xvec)\subset C_u$ and $E^s(\xvec)\subset C_s$. Once invariance has been established, strict inclusions upgrade to  
\[
E^u(\xvec)\setminus\{0\} \subset \interior C_u, \quad E^s(\xvec)\setminus\{0\} \subset \interior C_s.
\]
Since \(C_u\cap C_s=\{0\}\) and the dimensions add to three, \eqref{eq:bundle-direct-sum} follows.  The unit sections of the two global cones are disjoint compact subsets of the sphere, hence their positive angular separation gives \eqref{eq:uniform-transversality}. 
Estimate \eqref{eq:bundle-expansion} is a direct consequence of Corollary \ref{cor:iterated}. If \(w_s=A^{(n)}(\xvec)v_s\), then invariance gives \(w_s\in E^s(T^n\xvec)\subset C_s\). Applying the inverse estimate of Corollary \ref{cor:iterated} at \(T^n\xvec\) gives
\[
|v_s| =|[A^{(n)}(\xvec)]^{-1}w_s| \geq\Lambda_s^n|w_s|,
\]
which is \eqref{eq:bundle-contraction}.
\end{proof}

Once the uniformly hyperbolic splitting $\R^3 = E^s(\xvec)\oplus E^u(\xvec)$ has been established, we are in position to define the projectors
\begin{equation}\label{eq:projectors}
\begin{split}
    P^s(\xvec):\R^3\to E^s(\xvec), & \quad \ker P^s(\xvec) = E^u(\xvec),\\
    P^u(\xvec):\R^3\to E^u(\xvec), & \quad \ker P^u(\xvec) = E^s(\xvec).
\end{split}
\end{equation}
for any $\xvec\in\mathcal{R}$. By Proposition \ref{prop:measurable-splitting} they are measurable maps, and by uniform transversality,
\[
\esssup_{\xvec\in\T^3} (\|P^u(\xvec)\|+\|P^s(\xvec)\|)<\infty.
\]
We define them arbitrarily on the negligible complement of \(\mathcal R\).

\subsection{Conditional exponential growth}\label{s:conditional}

Once we have settled that the velocity field \eqref{eq:velocity} defines a uniformly hyperbolic map almost everywhere in $\T^3$, it follows that all initial data that is not strictly supported in the stable direction will grow exponentially fast in time. In this section we give details about this claim.

For any \(t\geq0\), the solution to the induction equation \eqref{eq:dynamo} is given by \eqref{eq:solutions} in terms of the flow map $\phi_t$ defined by $u$. In particular, a change of variables yields,
\[
B(t,\phi_t(\yvec)) = D\phi_t(\yvec)B_0(\yvec).
\]
Since \(\phi_t\) is volume-preserving and Lipschitz, given any $1\leq p\leq\infty$, it follows that if $B_0\in L^p(\T^3)$, then \(B(t,\cdot)\in L^p(\T^3)\). At integer times, we can directly look at iterations of the flow map $T$,
\[
B(n,T^n\yvec) =A^{(n)}(\yvec)B_0(\yvec),
\]
where we recall from the previous section that we set
\[
A(\xvec)=DT(\xvec), \quad A^{(n)}(\xvec) = A(T^{n-1}\xvec)\cdots A(T\xvec)A(\xvec) = DT^n(\xvec).
\]

First of all, we show that if the initial datum is divergence-free in the sense of distributions, the solution to the induction equation \eqref{eq:dynamo} remains divergence-free for all times.

\begin{lemma}\label{lem:divergence-preservation}
Let \(1\leq p\leq\infty\).  If \(B_0\in L^p(\T^3)\) and \(\div B_0=0\) in distributions, then so it is \(B(t,\cdot)\) as defined in \eqref{eq:solutions} for every \(t\geq0\).
\end{lemma}

\begin{proof}
For \(\varphi\in C^\infty(\T^3)\), change variables \(\xvec=\phi_t(\yvec)\) and use the chain rule,
\[
\int_{\T^3} B(t,\xvec)\cdot\nabla\varphi(\xvec)\dd{\xvec} = \int_{\T^3}D\phi_t(\yvec)B_0(\yvec) \cdot\nabla\varphi(\phi_t(\yvec))\dd{\yvec} = \int_{\T^3}B_0(\yvec)\cdot \nabla(\varphi\circ\phi_t)(\yvec)\dd{\yvec}=0.
\]
The last test function is Lipschitz. Smooth approximation and dominated convergence justify its use in the distributional identity.
\end{proof}

Without further ado, let us show that the uniform hyperbolic splitting provided by $u_\alpha$ yields that solutions to the induction equation starting from initial data that is not fully supported in the stable bundle will grow exponentially.

\begin{proposition}\label{prop:conditional-growth}
Let $P^u$ be the projector onto the measurable space $E^u(\xvec)$ provided by \eqref{eq:projectors} and Proposition \ref{prop:measurable-splitting}. For every \(1\leq p\leq\infty\) and every \(B_0\in L^p(\T^3)\), there holds
\begin{equation}\label{eq:integer-growth}
\|B(n,\cdot)\|_{L^p} \geq \kappa\Lambda_u^n\|P^uB_0\|_{L^p}, \quad \forall n\in\N.
\end{equation}
Constants $\kappa,\Lambda_u>0$ come from Proposition \ref{prop:measurable-splitting}.
Moreover, there exists \(C_0>0\) such that
\[
\|B(t,\cdot)\|_{L^p} \geq C_0\e^{(\log\Lambda_u)t}\|P^uB_0\|_{L^p}, \quad \forall t\in\R_+.
\]
\end{proposition}

\begin{proof}
Write \(B_0=B_0^u+B_0^s\), where \(B_0^{u}=P^{u}B_0\) and \(B_0^s=P^sB_0\). Uniform transversality \eqref{eq:uniform-transversality} and expansion \eqref{eq:bundle-expansion} give, for almost every \(\yvec\in\T^3\),
\[
|A^{(n)}(\yvec)B_0(\yvec)| \geq\kappa|A^{(n)}(\yvec)B_0^u(\yvec)| \geq\kappa\Lambda_u^n|B_0^u(\yvec)|.
\]
Since \(T\) preserves Lebesgue measure, taking the \(L^p\) norm proves the growth estimate for integer times \eqref{eq:integer-growth}. For the interval between successive integers, let
\[
m_\star = \inf_{0\leq\tau\leq1} \operatorname*{ess\,inf}_{\yvec\in\T^3} \sigma_{\min}(D\phi_\tau(\yvec)),
\]
where $\sigma_{\min}(A) = \min_{|v|=1}|Av|$ denotes the smallest singular value of a matrix $A$. The inverse maps are uniformly Lipschitz on one period. More precisely, setting
\[
L_\star =\sup_{0\leq\tau\leq1}\operatorname{Lip}(\phi_\tau^{-1})<\infty,
\]
then $\sigma_{\min}(D\phi_\tau)\geq L_\star^{-1}$ almost everywhere, and thus \(m_\star>0\). By volume preservation we get
\[
\|D\phi_\tau F\|_{L^p}\geq m_\star\|F\|_{L^p}, \quad 0\leq\tau\leq1.
\]
Time-periodicity of the flow gives that if \(t=n+\tau\), we can write $\phi_{n+\tau} = \phi_\tau \circ T^n$. Using the result for integer times it follows that
\[
\|B(t,\cdot)\|_{L^p} \geq m_\star \|DT^n B_0\|_{L^p} \geq m_\star\kappa\Lambda_u^n\|P^uB_0\|_{L^p} \geq\frac{m_\star\kappa}{\Lambda_u}e^{(\log\Lambda_u)t}\|P^uB_0\|_{L^p}.
\]
Hence we arrive to the claim of the proposition.
\end{proof}

\begin{remark}\label{rmk:alpha-dep}
From the proof of Proposition \ref{prop:conditional-growth} we see that we have an explicit expression for the multiplicative constant
\[
C_0 = \frac{m_\star\kappa}{\Lambda_u}.
\]
From Proposition \ref{prop:measurable-splitting} we get that $\kappa>0$ becomes independent of $\alpha$, from Corollary \ref{cor:iterated} we see that $\Lambda_u \asymp \alpha$, and moreover $m_\star$ scales like the smallest singular value $\sigma_{\min} \asymp\alpha^{-3}$. In particular we see that $C_0\asymp \alpha^{-4}$ becomes very small as $\alpha$ grows, as one should expect.
\end{remark}

Proposition \ref{prop:conditional-growth} gives an exact obstruction to the universal result claimed in Theorem \ref{thm}. In particular, for any $1\leq p\leq\infty$ we can define the set
\begin{equation}\label{eq:stable-kernel}
\mathcal Z_p = \left\{ B:\T^3\to\R^3\mid B\in L^p(\T^3), \, \div B=0,\, B(\xvec)\in E^s(\xvec)\text{ for a.e.\ }\xvec \right\}.
\end{equation}
We call $\mathcal{Z}_p$ the \emph{stable kernel}.
For any fixed \(p\), Proposition \ref{prop:conditional-growth} implies that every non-zero divergence-free data in \(L^p\) grow exponentially if and only if the stable kernel is trivial, i.e.\ \(\mathcal Z_p=\{0\}\). Indeed, if we show that $\mathcal{Z}_p=\{0\}$ we readily see that Theorem \ref{thm} follows. Conversely, if \(0\neq B_0\in\mathcal Z_p\), then \eqref{eq:bundle-contraction} and volume preservation give
\begin{equation}\label{eq:decay}
\|B(n,\cdot)\|_{L^p} \leq\Lambda_s^{-n}\|B_0\|_{L^p},
\end{equation}
and hence the solution decays exponentially. In the remainder of this paper we will focus on proving that indeed $\mathcal{Z}_p=\{0\}$.

\begin{remark}\label{rmk:zero-mean}
One important comment at this point is that we can already reduce the problem to study divergence-free initial data that are mean-free. Indeed, the mean is a conserved quantity by the equation \eqref{eq:dynamo},
\[
\int_{\T^3}B(t,\xvec)\dd\xvec = \int_{\T^3}B_0(\xvec)\dd\xvec =: M \quad \forall t\geq 0.
\]
If $B_0\in \mathcal{Z}_p$, the contraction estimate \eqref{eq:decay} gives
\[
|M| = \left|\int_{\T^3}B(n,\xvec)\dd\xvec \right| \leq |\T^3|^{1-1/p}\|B(n,\cdot)\|_{L^p} \leq |\T^3|^{1-1/p}\Lambda_s^{-n}\|B_0\|_{L^p} \to 0 \quad \text{as } n\to \infty.
\]
This yields a contradiction unless $M=0$, i.e.\ unless the initial datum is mean-free.
\end{remark}

\section{Triviality of the stable kernel and proof of Theorem \ref{thm}}\label{s:L2-kernel}

In this section we will show that the stable kernel $\mathcal{Z}_p$ from \eqref{eq:stable-kernel} must only contain the trivial vector, completing the proof of Theorem \ref{thm}. The key ingredient is a bunching inequality that states that the Lipschitz norm of $T$ is strictly smaller than the strength of contraction in the stable direction. First of all, let us introduce some tools that will be useful in order to prove that $\mathcal{Z}_p=\{0\}$. 

Let $G:\T^3\to\T^3$ be an orientation and volume-preserving, bi-Lipschitz map, and let $\psi$ be a smooth vector field in $\T^3$. We say that $\psi_G$ defined by 
\begin{equation}\label{eq:covariant-Piola-definition}
\psi_G(G(\xvec)) = DG(\xvec)^{-\top}\psi(\xvec),
\end{equation}
is the \emph{covariant Piola transform} of $\psi$ by $G$. 
Here $DG(\xvec)^{-\top}$ denotes the inverse-transport matrix. In differential geometry language, this is typically called the ordinary pullback of a $1$-form. This is the natural transformation law for covectors, and is characterized by the pointwise duality identity
\[
DG(\xvec)v\cdot\psi_G(G(\xvec)) = v\cdot\psi(\xvec)
\]
for all $v\in\R^3$.  We refer to \cite{MarsdenHughes94}*{Chapter 1.7} for further information about the Piola transform.

The following result proves that the covariant Piola transform commutes with the curl.

\begin{lemma}\label{lemma:piola-transform}
Let $G:\T^3\to\T^3$ be an orientation and volume-preserving, bi-Lipschitz map, that is differentiable almost everywhere, and let $\psi$ be a smooth vector field in $\T^3$. Then, we have
\begin{equation}\label{eq:covariant-Piola-curl}
(\curl\psi_G)(G(\xvec)) = DG(\xvec)\curl\psi(\xvec),
\end{equation}
in distributions, where $\psi_G$ is defined in \eqref{eq:covariant-Piola-definition}.
\end{lemma}

\begin{proof}
Given the map $G$, we set $M(\xvec)=DG(\xvec)$. Before proving \eqref{eq:covariant-Piola-curl}, let us first consider the following auxiliary identity
\begin{equation}\label{eq:auxiliary}
\curl (M^\top(\eta\circ G)) = M^{-1}(\curl\eta)\circ G
\end{equation}
for every smooth vector field $\eta\in C^\infty(\T^3)$, in distributions. Indeed, assuming \eqref{eq:auxiliary}, we see that \eqref{eq:covariant-Piola-curl} follows by integration by parts. Take $\eta\in C^\infty(\T^3)$, define $\yvec = G(\xvec)$, and compute
\[
\begin{split}
\int_{\T^3} \curl\psi_G \cdot\eta\dd\yvec & = \int_{\T^3} \psi_G\cdot\curl\eta\dd\yvec = \int_{\T^3} M(\xvec)^{-\top}\psi(\xvec)\cdot(\curl\eta)(G(\xvec))\dd\xvec \\
& = \int_{\T^3}\psi(\xvec)\cdot M(\xvec)^{-1}(\curl\eta)(G(\xvec))\dd\xvec.
\end{split}
\]
Now we use \eqref{eq:auxiliary} to obtain
\[
\begin{split}
\int_{\T^3} \curl\psi_G \cdot\eta\dd\yvec & = \int_{\T^3}\psi \cdot \curl (M^\top(\eta\circ G))\dd\xvec = \int_{\T^3} \curl\psi\cdot M^\top(\eta\circ G)\dd\xvec \\
& = \int_{\T^3} M\curl\psi\cdot (\eta\circ G)\dd\xvec = \int_{\T^3} (M\curl\psi)\circ G^{-1} \cdot \eta\dd\yvec
\end{split}
\]
From here it follows that $\curl\psi_G = (M\curl\psi)\circ G^{-1}$, i.e.\ \eqref{eq:covariant-Piola-curl}, in the sense of distributions. For completeness, we also prove \eqref{eq:auxiliary}. Following Einstein's convention for repeated indices we write
\[
M^\top(\eta\circ G) = (\eta_i\circ G)\nabla G_i.
\]
Now, since $\curl(f\nabla g) = \nabla f \times\nabla g$, the chain rule yields
\[
\begin{split}
\curl(M^\top(\eta\circ G)) & = ((\partial_j\eta_i)\circ G)\nabla G_j\times\nabla G_i\\
& = ((\curl \eta)_1\circ G)\nabla G_2\times\nabla G_3 + ((\curl \eta)_2\circ G)\nabla G_3\times\nabla G_1 \\
& \quad + ((\curl \eta)_3\circ G)\nabla G_1\times\nabla G_2.
\end{split}
\]
Because the rows of the matrix $M$ are $\nabla G_i$ with $1\leq i\leq 3$, and because $\det M = 1$, it follows that for any vector $w\in\R^3$ we can write
\[
M^{-1}w = w_1\nabla G_2\times\nabla G_3 + w_2\nabla G_3\times\nabla G_1 + w_3\nabla G_1\times\nabla G_2.
\]
So in particular $\curl(M^\top(\eta\circ G)) = M^{-1}(\curl\eta)\circ G$, and we get to the identity \eqref{eq:auxiliary}.
\end{proof}

With this tool in our bag, we can proceed with the central lemma of the section, that proves $\mathcal{Z}_p=\{0\}$ provided that a certain strict inequality holds true.

\begin{lemma}\label{lem:dual-PDE-exclusion}
Let \(T:\T^3\to\T^3\) be the map defined in \eqref{eq:shears}.  Let \(E^s(\xvec)\) be the measurable stable direction from Proposition \ref{prop:measurable-splitting}, which, we recall, has the property,
\[
|DT^n(\xvec)v| \leq\Lambda_s^{-n}|v| \quad \forall v\in E^s(\xvec),\ \forall n\geq0.
\]
at every point $\xvec\in\mathcal{R}$.  If there holds
\begin{equation}\label{eq:PDE-bunching}
\operatorname{Lip}(T)<\Lambda_s,
\end{equation}
then the only mean-free \(B_0\in L^p(\T^3)\), with fixed $1\leq p\leq \infty$, satisfying $\div B_0=0$ in distributions, and such that $B_0(\xvec)\in E^s(\xvec)$ for almost every $\xvec\in\T^3$, is the trivial vector field $B_0=0$.
\end{lemma}

For clarification, we recall that the Lipschitz norm of $T$ is defined by
\[
\Lip(T) = \esssup_{\xvec\in\T^3}\|DT(\xvec)\| = \esssup_{\xvec\in\T^3} \sup_{|v|=1}|DT(\xvec)v|,
\]
As it is the convention along this paper, the un-subscripted norm $\|\cdot\|$ denotes the operator norm.

\begin{proof}
Let $B_0\in L^p$ be as described in the claim of the lemma. In particular it is mean-free, and by Remark \ref{rmk:zero-mean} there holds that
\[
B_n = (DT^nB_0)\circ T^{-n}
\]
remains mean-free for all $n\geq 0$. Since $B_n$ is mean-free and $\div B_n = 0$, we can define the Coulomb vector potential
\begin{equation}\label{eq:Coulomb-potential}
A_n=\curl(-\Delta)^{-1}B_n,
\end{equation}
or equivalently $B_n = \curl A_n$.
The periodic convolution kernel \(K\) of \(\curl(-\Delta)^{-1}\) satisfies $|K(\xvec)|\lesssim |\xvec|^{-2}$ near the origin and is smooth away from it, see e.g.\ \cite{Aubin82}.  In particular, \(K\in L^1(\T^3)\). Young's inequality therefore gives, for any $1\leq p\leq\infty$,
\begin{equation}\label{eq:Coulomb-Lp-estimate}
\|A_n\|_{L^p}\lesssim \|B_n\|_{L^p}.
\end{equation}
Moreover, if $1<p<\infty$, one has $A_n\in W^{1,p}(\T^3)$.
Next we fix an arbitrary vector field $\psi_0\in C^\infty(\T^3)$, and define
\[
\psi_n = ((DT^n)^{-\top}\psi_0)\circ T^{-n},
\]
i.e.\ $\psi_n$ is the covariant Piola transform of $\psi_0$ by the bi-Lipschitz mapping $T^n$, or equivalently, the vector $\psi_0$ is transported covariantly by the flow $T^n$. As pointed out before, here $DT^n$ means $D(T^n)$, we just avoid introducing unnecessary parenthesis for a better presentation. Using the definitions of $B_n$ and $\psi_n$ together with the volume-preservation of $T$ it readily follows that
\begin{equation}\label{eq:transported-duality}
\int_{\T^3} B_n\cdot\psi_n\dd\xvec = \int_{\T^3} DT^nB_0 \cdot (DT^n)^{-\top}\psi_0\dd\xvec = \int_{\T^3} B_0\cdot\psi_0\dd\xvec
\end{equation}
remains constant for all $n\geq 0$. Identity \eqref{eq:transported-duality} expresses the conservation of the pairing between the vector and covector evolutions\footnote{If $\psi_0=A_0$, covariant transport gives $\curl\psi_n=B_n$. Hence, $\psi_n$ is a generally non-Coulomb vector potential of $B_n$, and \eqref{eq:transported-duality} represents the conservation of the magnetic helicity. Nonetheless, for this argument we require $\psi_0$ to be smooth so that we can use it as a test vector.}. Putting this relation together with the definition of the Coulomb potential \eqref{eq:Coulomb-potential} and integrating by parts we find
\[
\int_{\T^3} B_0\cdot\psi_0\dd\xvec= \int_{\T^3}B_n\cdot\psi_n\dd\xvec = \int_{\T^3}\curl A_n\cdot\psi_n\dd\xvec = \int_{\T^3}A_n\cdot\curl\psi_n\dd\xvec.
\]
This integration by parts is rigorous for $1<p<\infty$ since $A_n\in W^{1,p}(\T^3)$. For the endpoints $p=1$ and $p=\infty$ it follows right away after a suitable mollification.
Hölder's inequality and \eqref{eq:Coulomb-Lp-estimate} yield
\begin{equation}\label{eq:intermidiate-estimate}
\left|\int_{\T^3} B_0\cdot\psi_0\dd\xvec\right|\lesssim \|A_n\|_{L^p}\|\curl\psi_n\|_{L^{q}} \lesssim \|B_n\|_{L^p}\|\curl\psi_n\|_{L^{q}},
\end{equation}
where $q$ denotes the Hölder's conjugate of $p$. On the one hand, since by assumption $B_0(\xvec)\in E^s(\xvec)$ for almost every $\xvec\in\T^3$, Proposition \ref{prop:measurable-splitting} gives that $B_n(\xvec)\in E^s(\xvec)$ for all $n\geq 0$, and thus
\[
\|B_n\|_{L^p}\leq \Lambda_s^{-n}\|B_0\|_{L^p}.
\]
On the other hand, $\psi_n$ is a covariant Piola transform, and Lemma \ref{lemma:piola-transform} yields that 
\[
\curl\psi_n(T^n(\xvec)) = \curl \bigl(DT^n(\xvec)^{-\top}\psi_0(\xvec)\bigr) = DT^n(\xvec)\curl\psi_0(\xvec).
\]
Therefore, since $T^n$ is volume-preserving it follows
\[
\|\curl\psi_n\|_{L^q} \leq \esssup_{\xvec\in\T^3}\|DT^n(\xvec)\|\|\curl\psi_0\|_{L^q} = \Lip(T^n)\|\curl\psi_0\|_{L^q}.
\]
Additionally, we have the estimate $\Lip(T^n)\leq \Lip(T)^n$. Indeed, for every $\xvec\in\mathcal{R}$ we can write
\[
DT^n(\xvec)= D(T^n)(\xvec) = DT(T^{n-1}\xvec)\hdots DT(T\xvec)DT(\xvec),
\]
and thus, by submultiplicativity of the operator norm
\[
\|DT^n(\xvec)\| \leq \prod_{k=0}^{n-1}\|DT(T^k\xvec)\| \leq \left(\esssup_{\yvec\in\T^3} \|DT(\yvec)\|\right)^n,
\]
so there follows $\Lip(T^n)\leq \Lip(T)^n$. Combining all, we arrive at the estimate
\begin{equation}\label{eq:curl-dual-growth}
\|\curl\psi_n\|_{L^q} \leq \Lip (T)^n\|\curl\psi_0\|_{L^q}.
\end{equation}
Therefore, introducing both estimates in \eqref{eq:intermidiate-estimate} we get
\[
\left|\int_{\T^3} B_0\cdot\psi_0\dd\xvec\right|\lesssim \left(\frac{\Lip(T)}{\Lambda_s}\right)^n\|B_0\|_{L^p}\|\curl\psi_0\|_{L^q},
\]
which converges to zero as $n\to\infty$, assuming that the bunching condition from the claim of the lemma $\Lip(T)<\Lambda_s$ holds true. This gives that
\[
\int_{\T^3} B_0\cdot\psi_0\dd\xvec=0
\]
for any $\psi_0\in C^\infty(\T^3)$ and thus $B_0=0$.
\end{proof}

At this point we can put together all the pieces from the previous sections and conclude with the proof of the main theorem.

\begin{proof}[Proof of Theorem \ref{thm}]
The velocity field that we introduced in \eqref{eq:velocity} defines a time-one map $T$ that satisfies a uniform cone condition according to Proposition \ref{prop:cones}. With this, Proposition \ref{prop:measurable-splitting} yields that there is a measurable uniform hyperbolic splitting $\R^3 = E^s(\xvec)\oplus E^u(\xvec)$ for almost every $\xvec\in\T^3$, and hence there is a well-defined measurable projector $P^u$ onto the subspace $E^u(\xvec)$. Proposition \ref{prop:conditional-growth} then shows that solutions to the induction equation \eqref{eq:dynamo} satisfy
\[
\|B(t,\cdot)\|_{L^p} \geq C_0\e^{(\log\Lambda_u)t}\|P^uB_0\|_{L^p},
\]
i.e.\ they grow exponentially fast in $L^p$ for any $1\leq p\leq \infty$, provided that the initial datum satisfies $\|P^uB_0\|_{L^p}\neq 0$. Finally, Lemma \ref{lem:dual-PDE-exclusion} yields that the stable kernel $\mathcal{Z}_p=\{0\}$, which precisely implies that $\|P^uB_0\|_{L^p}>0$, provided that the time-one map satisfies the bunching inequality
\[
\Lip(T)<\Lambda_s.
\]
This is the last point that remains to be checked. Proposition \ref{prop:measurable-splitting} gives that the contraction rate $\Lambda_s$ along the stable direction $E^s(\xvec)$ is of order $\alpha^3$, meaning that there exists a constant $\delta>0$ such that $\Lambda_s\geq \delta\alpha^3$. Additionally, we can bound the Lipschitz norm of $T$ using its Frobenius norm, so using the definition \eqref{eq:gradT} we get
\[
\Lip(T) = \esssup_{\xvec\in\T^3}\sup_{|v|\leq 1} |DT(\xvec)v| \leq \left( \sum_{i,j}|DT(\xvec)_{i,j}|^2 \right)^{1/2} \leq \sqrt{\alpha^4+3\alpha^2+3},
\]
see also Remark \ref{rmk:lipT}. In particular
\[
\frac{\sqrt{\alpha^4+3\alpha^2+3}}{\delta\alpha^3} < 1
\]
holds true for all $\alpha$ sufficiently large, and thus the claim of the theorem follows.
\end{proof}

\begin{remark}
The key ingredient to show that there are no divergence-free vector fields in the stable bundle is the strict inequality $\Lip(T)<\Lambda_s$ from Lemma \ref{lem:dual-PDE-exclusion}. This holds true for large values of $\alpha$ since $\Lambda_s\asymp \alpha^3$ and $\Lip(T)\asymp\alpha^2$, and comes from the covariant Piola transform and estimate \eqref{eq:curl-dual-growth}. However, by construction, $\psi=\psi_n$ solves (at integer times) the covariant transport equation 
\[
\partial_t\psi + (u\cdot\nabla)\psi = -(\nabla u)^\top\psi.
\]
This on the other hand implies that $\curl\psi_n$ solves the induction equation \eqref{eq:dynamo}. Hence, one could naively try to develop $L^q$ energy estimates and get something similar to \eqref{eq:curl-dual-growth} directly from the PDE. However, via energy estimates we can only get
\[
\|\curl\psi_n\|_{L^q} \leq \e^{c_q\|\nabla u\|_{\infty} n}\|\curl\psi_0\|_{L^q},
\]
for some $c_q>0$, and where $\|\nabla u\|_{\infty}$ denotes the $L^\infty$ norm of $\nabla u$ both in time and space. According to \eqref{eq:velocity} we see that $\|\nabla u\|_\infty\asymp \alpha$, thus this estimate yields a prefactor that is exponential in $\alpha$, and that can never be controlled by $\Lambda_s\asymp\alpha^3$ for $\alpha$ large.
\end{remark}

\subsection{Non-coercivity of the unstable projector}\label{s:non-uniformity}

The proof of the main theorem is now complete. However, before finalising this note, we include a brief comment about the fact that, even if the unstable projection is injective on divergence-free fields, as follows from \(\mathcal Z_p=\{0\}\), it is not bounded below. In particular, from the arguments in this paper it does not follow that we can obtain universal growth estimates of the form
\[
\|B(t,\cdot)\|_{L^p} \gtrsim \e^{\lambda t}\|B_0\|_{L^p}.
\]
We could get estimates like this via Proposition \ref{prop:conditional-growth} if the unstable projector was coercive in the sense: there is a constant $\theta>0$ such that $\|P^uB\|_{L^p}\geq\theta\|B\|_{L^p}$ for all $B\in L^p$ with $\div B = 0$. This is a very strong result that does not happen to be true in this case, as the following lemma shows.

\begin{lemma}\label{lemma:non-coercivity}
For every \(1\leq p\leq\infty\), there exists a sequence $\{B_N^{(p)}\}_{N\in\N}\subset L^\infty(\T^3)$ such that $\div B_N^{(p)}=0$ in distributions, $\|B_N^{(p)}\|_{L^p}=1$ for all $N\in\N$, and
\[
\|P^uB_N^{(p)}\|_{L^p}\to 0, \quad \text{as } N\to \infty.
\]
\end{lemma}

\begin{proof}
We will prove that there exist constants \(C>0\) and \(\rho\in(0,1)\),
independent of \(p\) and \(N\), such that
\[
\|P^uB_N^{(p)}\|_{L^p}\leq C\rho^N.
\]
Let $e_1=(1,0,0)$ as before and define, almost everywhere,
\[
V_N(\xvec) = [DT^N(\xvec)]^{-1}e_1.
\]
The values of \(V_N\) on the negligible set on which \(DT^N\) is not defined may be chosen arbitrarily. Since \(T^N\) is bi-Lipschitz, \(V_N\in L^\infty(\T^3)\) and \(V_N\neq0\) almost everywhere.
Moreover, \(V_N\) is divergence-free. Indeed, for every \(\varphi\in C^\infty(\T^3)\), volume preservation and the change of variables \(\yvec=T^N(\xvec)\) give
\[
\int_{\T^3}V_N(\xvec)\cdot\nabla\varphi(\xvec)\dd\xvec = \int_{\T^3} [DT^N(T^{-N}\yvec)]^{-1}e_1 \cdot\nabla\varphi (T^{-N}\yvec) \dd\yvec = \int_{\T^3} DT^{-N}(\yvec)e_1\cdot \nabla\varphi (T^{-N}\yvec) \dd\yvec.
\]
By the chain rule we have
\[
DT^{-N}(\yvec)^\top \nabla\varphi (T^{-N}\yvec) = \nabla(\varphi\circ T^{-N})(\yvec),
\] 
thus integrating by parts
\[
\int_{\T^3}V_N(\xvec)\cdot\nabla\varphi(\xvec)\dd\xvec = \int_{\T^3} e_1\cdot \nabla(\varphi\circ T^{-N})(\yvec)\dd\yvec = 0,
\]
and we see that $\div V_N=0$ in distributions. Consider the finite-time stable line
\[
\ell_N(\xvec) = [DT^N(\xvec)]^{-1}\spann\{e_1\}.
\]
The stable graph-transform construction in the proof of Proposition~\ref{prop:measurable-splitting} gives constants \(C_1>0\) and \(\rho\in(0,1)\) such that
\begin{equation}\label{eq:finite-stable-line-convergence}
\esssup_{\xvec\in\T^3} d_{\mathrm{Gr}}(\ell_N(\xvec),E^s(\xvec)) \leq \sum_{k=N}^\infty q^k = \frac{q^N}{1-q} =: C_1\rho^N.
\end{equation}
Let \(\Pi_N(\xvec)\) and \(\Pi^s(\xvec)\) denote the orthogonal projections onto \(\ell_N(\xvec)\) and \(E^s(\xvec)\), respectively. \eqref{eq:finite-stable-line-convergence} implies that the distance between the projectors in the operator norm shrinks exponentially,
\[
\|\Pi_N(\xvec)-\Pi^s(\xvec)\| \leq C_1\rho^N.
\]
Since \(V_N(\xvec)\in\ell_N(\xvec)\), we have $V_N-\Pi^sV_N=(\Pi_N-\Pi^s)V_N$. Moreover, since \(P^u\) vanishes on \(E^s\) it follows that
\[
|P^u(\xvec)V_N(\xvec)| = |P^u(\xvec)(V_N(\xvec)-\Pi^s(\xvec)V_N(\xvec))| \leq \|P^u(\xvec)\|\|\Pi_N(\xvec)-\Pi^s(\xvec)\||V_N(\xvec)|.
\]
Uniform transversality of \(E^u\) and \(E^s\) implies that
\[
\operatorname*{ess\,sup}_{\xvec\in\T^3}\|P^u(\xvec)\|<\infty.
\]
Combining all, we obtain the pointwise estimate $|P^uV_N|\leq C\rho^N|V_N|$ almost everywhere. For a fixed \(p\in[1,\infty]\), define
\[
B_N^{(p)} = \frac{V_N}{\|V_N\|_{L^p}}.
\]
Then \(B_N^{(p)}\) is divergence-free, \(\|B_N^{(p)}\|_{L^p}=1\), and $\|P^uB_N^{(p)}\|_{L^p} \leq C\rho^N$.
\end{proof}

Lemma \ref{lemma:non-coercivity} shows that the identity
\[
\mathcal Z_p = \left\{ B\in L^p(\T^3): \div B=0,\ P^uB=0 \right\} = \{0\}
\]
is an injectivity statement rather than a coercive estimate. This explains why the data-dependent factor \(\|P^uB_0\|_{L^p}\) in Proposition~\ref{prop:conditional-growth} cannot be replaced by a uniform multiple of \(\|B_0\|_{L^p}\).

\section{Exponential decay and proof of Theorem \ref{thm:reversed-flow-decay}}\label{s:reversed-flow}

The preceding results are strongly asymmetric with respect to time reversal.  The original time-one map has a one-dimensional stable bundle, which we just saw it cannot contain any non-zero divergence-free \(L^p\) field.  By contrast, the stable bundle of the inverse map is the two-dimensional space \(E^u\), and we will show in this section that it does contain a non-zero bounded divergence-free field.  Consequently, reversing the order and the sign of the three shears produces a velocity field for which a nontrivial solution of the induction equation decays exponentially in \(L^\infty\).

As introduced before, we define the reversed velocity on one period by
\[
\widetilde u_\alpha(t,\xvec)=-u_\alpha(1-t,\xvec), \quad \text{for all } 0\leq t<1,
\]
and extend it one-periodically.  Equivalently, if \(\tau=t-\lfloor t\rfloor\), then, away from the switching times,
\[
\widetilde u_\alpha(t,x,y,z) =-3\alpha
\begin{cases}
(0,0,h(x)),&0\leq\tau<1/3,\\
(0,h(z),0),&1/3\leq\tau<2/3,\\
(h(y),0,0),&2/3\leq\tau<1.
\end{cases}
\]
Thus the three inverse shears are applied in reverse order, and the time-one map of \(\widetilde u_\alpha\) is \(T^{-1}\). The main purpose of this section is to prove Theorem \ref{thm:reversed-flow-decay}, and for that, the central point is to construct a bounded divergence-free field that lies in \(E^u\) almost everywhere. Here the result follows from the piecewise-affine structure of \(T\), which makes sufficiently small local unstable discs exactly flat.

As before, we let \(\mathcal S\) be the singular set of \(T\). The singular set of \(T^{-1}\) is \(T(\mathcal S)\), so we put $\Gamma=\mathcal S\cup T(\mathcal S)$. This is a finite union of Lipschitz hypersurfaces. Hence there exists a constant \(C_\Gamma>0\) such that
\begin{equation}\label{eq:tubular-neighbourhood-Gamma}
\bigl|\{\xvec\in\T^3: \dist(\xvec,\Gamma)<\delta\}\bigr| \leq C_\Gamma\delta
\end{equation}
for all \(0<\delta\leq 1\).

It will be convenient to use the graph-transform definition of the unstable plane from Section \ref{s:hyperbolic-splitting} also at points that have a regular past but whose forward orbit might meet \(\mathcal S\). Fix \(F_\star\in\mathscr G^u\), as in the proof of Proposition \ref{prop:measurable-splitting}, and define
\begin{equation}\label{eq:past-unstable-plane}
\widehat E^u(\xvec) =\lim_{m\to\infty}DT^m(T^{-m}\xvec)F_\star
\end{equation}
whenever the full backward derivative itinerary is defined, which is a full-measure set.  The contraction of the graph transform proves that this limit exists. On \(\mathcal R\), it
agrees with the bundle \(E^u\) from Proposition~\ref{prop:measurable-splitting}.  In particular, \(\widehat E^u(\xvec)\) depends only on the past sequence of affine branches
visited by \(\xvec\).

First, we need a technical lemma that shows that a local unstable disc consists of nearby points that share the same past branch itinerary.

\begin{lemma}\label{lem:flat-unstable-plaques}
There is a measurable function \(r:\mathcal R\to[0,\infty)\), positive almost everywhere, with the following property.  If \(r(\xvec)>0\), \(\vvec\in E^u(\xvec)\), and \(|\vvec|<r(\xvec)\), then
\begin{equation}\label{eq:affine-backward-orbits}
T^{-n}(\xvec+\vvec)=T^{-n}\xvec+DT^{-n}(\xvec)\vvec \quad \text{for every } n\geq 0.
\end{equation}
Moreover, the backward orbits of the two points encounter the same sequence of affine branches and 
\begin{equation}\label{eq:constant-past-unstable-plane}
\widehat E^u(\xvec+\vvec)=E^u(\xvec).
\end{equation}
\end{lemma}

\begin{proof}
Set \(\eta=\Lambda_u^{-1}\in(0,1)\), and choose \(\beta\in(\eta,1)\).  Since \(T\) preserves Lebesgue measure, estimate \eqref{eq:tubular-neighbourhood-Gamma} gives
\[
\sum_{n=0}^\infty\bigl|\{\xvec\in\T^3: \dist(T^{-n}\xvec,\Gamma)<\beta^n\}\bigr| = \sum_{n=0}^\infty \bigl|\{\xvec\in\T^3: \dist(\xvec,\Gamma)<\beta^n\}\bigr| \leq C_\Gamma\sum_{n=0}^\infty\beta^n<\infty.
\]
The Borel--Cantelli lemma therefore implies that, for almost every \(\xvec\),
\begin{equation}\label{eq:BC-distance}
\dist(T^{-n}\xvec,\Gamma)\geq\beta^n
\end{equation}
for every sufficiently large \(n\). After removing the null set $\cup_{k\in\Z}T^k(\Gamma)$, define
\begin{equation}\label{eq:plaque-radius}
r(\xvec)=\frac14\min\left\{\inf_{n\geq0}\eta^{-n}\dist(T^{-n}\xvec,\Gamma),r_\star\right\},
\end{equation}
where \(r_\star>0\) is a fixed number, that we choose smaller than one quarter of the injectivity radius of the flat torus.  For each \(n\), all additions, segments, and affine identities below are understood in the unique Euclidean lift centred at \(T^{-n}\xvec\). We extend \(r\) by zero on the exceptional set where it has not been defined. This is a measurable function.  Using \eqref{eq:BC-distance} we get that
\[
\eta^{-n}\dist(T^{-n}\xvec,\Gamma) \geq\left(\frac{\beta}{\eta}\right)^n\to\infty
\]
for all sufficiently large \(n\), while each of the finitely many remaining
terms is positive.  Thus \(r(\xvec)>0\) for almost every \(\xvec\).
Now fix such an \(\xvec\), take \(\vvec\in E^u(\xvec)\) with \(|\vvec|<r(\xvec)\), and set \(\yvec=\xvec+\vvec\).  We prove \eqref{eq:affine-backward-orbits} by induction.  For $n=0$ it is trivial, so suppose it holds at time \(n\).  By Proposition \ref{prop:measurable-splitting} and Corollary \ref{cor:iterated} it follows that
\begin{equation}\label{eq:inverse-contraction-on-Eu}
|DT^{-n}(\xvec)\vvec| \leq \Lambda_u^{-n}|\vvec|
\end{equation}
for every \(\vvec\in E^u(\xvec)\) and \(n\geq0\). Indeed, if \(\zvec=T^{-n}\xvec\) and \(w=DT^{-n}(\xvec)\vvec\), then \(w\in E^u(\zvec)\) and \(\vvec=DT^n(\zvec)w\), so the forward expansion estimate gives the claim. Putting this together with \eqref{eq:plaque-radius} there yields
\[
|T^{-n}\yvec-T^{-n}\xvec| =|DT^{-n}(\xvec)\vvec|\leq\eta^n|\vvec|<\frac14\dist(T^{-n}\xvec,\Gamma).
\]
Consequently, the segment joining \(T^{-n}\xvec\) and \(T^{-n}\yvec\) lies in one component of \(\T^3\setminus\Gamma\). The map \(T^{-1}\) is affine on this component, and therefore
\[
T^{-(n+1)}\yvec-T^{-(n+1)}\xvec =DT^{-1}(T^{-n}\xvec)\bigl(T^{-n}\yvec-T^{-n}\xvec\bigr) =DT^{-(n+1)}(\xvec)\vvec.
\]
This closes the induction.  It also shows that the two backward orbits visit the same affine branch at every step, so the finite graph transforms in \eqref{eq:past-unstable-plane} are identical for \(\xvec\) and \(\yvec\). Passing to the limit gives \eqref{eq:constant-past-unstable-plane}.
\end{proof}

For every point to which Lemma~\ref{lem:flat-unstable-plaques} applies, we may therefore introduce the flat disc
\begin{equation}\label{eq:flat-plaque}
\mathcal P_{\xvec} =\{\xvec+\vvec: \vvec\in E^u(\xvec),\ |\vvec|<r(\xvec)\}.
\end{equation}
The plane \(\widehat E^u\) is constant along this disc.  We next select a positive-measure family of discs and use two-dimensional stream functions on them to construct the desired vector field.

\begin{proposition}\label{prop:bounded-section-Eu}
There exists a non-zero, mean-free field $B_0\in L^\infty(\T^3)$ such that $\div B_0=0$ in distributions and $B_0(\xvec)\in E^u(\xvec)$ for almost every \(\xvec\in\T^3\).
\end{proposition}

\begin{proof}

By Lemma \ref{lem:flat-unstable-plaques}, after extending it by zero on a negligible set, the function \(r(\xvec)>0\) is measurable and positive almost everywhere on \(\T^3\). For every \(m\in\N\), define the positive level set
\[
A_m =\{\xvec\in\mathcal R:r(\xvec)\geq m^{-1}\}.
\]
Since
\[
\bigcup_{m=1}^{\infty}A_m = \{x\in\mathcal R:r(x)>0\}
\]
has full measure, there exists \(m_0\in\N\) such that $|A_{m_0}|\equiv \mathscr L^3(A_{m_0})>0$, where now we make explicit that $\mathscr L^d$ is the $d$-dimensional Lebesgue measure. Writing \(\xvec=(\tau,y,z)\in\T\times\T^2\), Fubini's theorem gives
\[
\mathscr L^3(A_{m_0}) = \int_{\T^2} \mathscr L^1 ( \{\tau\in\T:(\tau,y,z)\in A_{m_0}\} )\dd y\dd z,
\]
Consequently, there exist \(y_0,z_0\in\T\) such that the section $E=\{\tau\in\T:(\tau,y_0,z_0)\in A_{m_0}\}$ has positive one-dimensional measure. Covering \(\T\) by finitely many open coordinate intervals, we may choose one such interval \(I\subset\T\), whose closure is contained in a coordinate chart, and for which $\mathscr L^1(E\cap I)>0$.  By inner regularity of Lebesgue measure, there exists a compact set \(K\subset E\cap I\) with \(\mathscr L^1(K)>0\). Set $r_0=m_0^{-1}$, $\xvec_\tau=(\tau,y_0,z_0)$ with $\tau\in K$. Then, $\xvec_\tau\in\mathcal R$ and $r(\xvec_\tau)\geq r_0$ for every \(\tau\in K\).

The unstable planes are uniformly transverse to \(e_1=(1,0,0)\). Consequently, for \(\tau\in K\), there is a uniformly bounded vector \(g(\tau)=(g_1(\tau),g_2(\tau))\in\R^2\) such that
\[
E^u(\xvec_\tau)=\left\{(g(\tau)\cdot\boldsymbol{s},s_1,s_2):\boldsymbol{s}=(s_1,s_2)\in\R^2\right\}.
\]
Choose \(R_0>0\) sufficiently small that $|(g(\tau)\cdot\boldsymbol{s},s_1,s_2)|<r_0$ whenever \(\tau\in K\) and \(|\boldsymbol{s}|<R_0\).  By Lemma~\ref{lem:flat-unstable-plaques}, each of the sets
\begin{equation}\label{eq:uniform-flat-plaques}
\mathcal P_\tau =\{ (\tau+g(\tau)\cdot\boldsymbol{s}, y_0+s_1,z_0+s_2): |\boldsymbol{s}|<R_0\}
\end{equation}
is contained in the set \eqref{eq:flat-plaque} through \(\xvec_\tau\). 
Choose coordinate intervals \(J_y,J_z\subset\T\) containing \(y_0,z_0\), respectively. Since \(K\) is compactly contained in \(I\) and \(g\) is uniformly bounded on \(K\), after decreasing \(R_0\) if necessary, all the discs \(P_\tau\), with \(\tau\in K\), are contained in the single product coordinate chart $I\times J_y\times J_z\subset\T^3$. We fix a lift of this chart to \(\R^3\), and all differences and affine identities in the remainder of the proof are understood in this lift.

We claim that these discs are pairwise disjoint.  If \(\mathcal P_\tau\) and \(\mathcal P_{\tau'}\) had a common point, the backward orbits of that point would encounter the same sequence of affine branches as both centres.  By the past graph-transform formula \eqref{eq:past-unstable-plane}, this would imply
\[
E^u(\xvec_\tau)=E^u(\xvec_{\tau'}).
\]
The difference \((\tau'-\tau)e_1\) of the two centres would then belong to this common plane.  Uniform transversality of \(E^u\) to \(e_1\) forces \(\tau=\tau'\), proving the claim.

The non-intersection also gives regularity of the slope of the vector \(g\).  Let \(\tau<\tau'\) be in \(K\), and put \(d=g(\tau')-g(\tau)\).  If
\[
|d|>\frac{\tau'-\tau}{R_0},
\]
then
\[
\boldsymbol{s} =-\frac{\tau'-\tau}{|d|^2}d
\]
satisfies \(|\boldsymbol{s}|<R_0\) and $\tau'-\tau+d\cdot\boldsymbol{s}=0$. The two points in \eqref{eq:uniform-flat-plaques} having this same value of \(\boldsymbol{s}\) would coincide, which is impossible.  Hence
\[
|g(\tau')-g(\tau)| \leq R_0^{-1}|\tau'-\tau| \quad \text{for every } \tau,\tau'\in K.
\]
Replacing \(I\) by \([\min K,\max K]\), we may assume that the endpoints of \(I\) belong to \(K\).
We can extend \(g\) linearly across the complementary intervals of \(K\) to obtain a Lipschitz map on \(I\), still denoted by \(g\).  We choose a disc \(D\subset \interior\{\boldsymbol{s}\in\R^2: |\boldsymbol{s}|<R_0\}\), centred at the origin, so small that
\begin{equation}\label{eq:small-box-condition}
\Lip(g)\sup_{\boldsymbol{s}\in D}|\boldsymbol{s}|\leq\frac{1}{2},
\end{equation}
and we define
\[
\Phi(\boldsymbol{s},\tau) =\bigl(\tau+g(\tau)\cdot\boldsymbol{s}, y_0+s_1,z_0+s_2\bigr), \quad (\boldsymbol{s},\tau)\in D\times I.
\]
Using \eqref{eq:small-box-condition} we readily see that this map is bi-Lipschitz. We can define now two vectors that are tangent to the unstable by
\begin{equation}\label{eq:plaque-tangent-vectors}
X_1(\tau)=\partial_{s_1}\Phi=(g_1(\tau),1,0), \quad
X_2(\tau)=\partial_{s_2}\Phi=(g_2(\tau),0,1),
\end{equation}
and the Jacobian, defined almost everywhere, is given by
\[
J(\boldsymbol{s},\tau) =\det D\Phi(\boldsymbol{s},\tau) =1+g'(\tau)\cdot\boldsymbol{s}.
\]
In particular, \eqref{eq:small-box-condition} gives $1/2\leq J(\boldsymbol{s},\tau)\leq 3/2$ for almost every \((\boldsymbol{s},\tau)\in D\times I\). Finally, choose a nonconstant streamfunction \(\psi\in C_c^\infty(D)\), and define \(B_0\) by
\begin{equation}\label{eq:bounded-section-construction}
B_0(\Phi(\boldsymbol{s},\tau)) =\frac{\mathbf 1_K(\tau)}{J(\boldsymbol{s},\tau)} \bigl[
\partial_{s_2}\psi(\boldsymbol{s})X_1(\tau) -\partial_{s_1}\psi(\boldsymbol{s})X_2(\tau) \bigr]
\end{equation}
for \((\boldsymbol{s},\tau)\in D\times I\). Additionally we set \(B_0=0\) outside \(\Phi(D\times I)\).  The bounds on \(g\), \(J^{-1}\), and \(\nabla\psi\) show that \(B_0\in L^\infty(\T^3)\).  It is non-zero because \(K\) has positive measure and \(\psi\) is chosen nonconstant. For \(\tau\in K\), \(X_1(\tau)\) and \(X_2(\tau)\) span the unstable plane along the corresponding disc.  Since \(\Phi\) is bi-Lipschitz and \(\mathcal R^c\) has zero three-dimensional measure, it follows that
\[
B_0(\xvec)\in E^u(\xvec) \quad \text{for almost every } \xvec\in\T^3.
\]

It remains to verify the divergence constraint.  Let \(\varphi\in C^\infty(\T^3)\).  The change of variables \(\xvec=\Phi(\boldsymbol{s},\tau)\), \eqref{eq:plaque-tangent-vectors}, and the factor \(J^{-1}\) in \eqref{eq:bounded-section-construction} give
\[
\int_{\T^3}B_0(\xvec)\cdot\nabla\varphi(\xvec)\dd\xvec =\int_K\int_D \bigl[ \partial_{s_2}\psi(\boldsymbol{s}) \partial_{s_1}(\varphi\circ\Phi)(\boldsymbol{s},\tau) -\partial_{s_1}\psi(\boldsymbol{s}) \partial_{s_2}(\varphi\circ\Phi)(\boldsymbol{s},\tau) \bigr]\dd\boldsymbol{s}\dd\tau =0.
\]
The last equality follows by integration by parts in \(\boldsymbol{s}\) since the two mixed derivatives cancel.  Thus \(\div B_0=0\) in distributions.  Notice that the possibly irregular factor \(\mathbf 1_K(\tau)\) causes no additional term, because the weak calculation differentiates only in the two directions tangent to the discs. Finally, since \(X_1(\tau)\) and \(X_2(\tau)\) do not depend on \(\boldsymbol{s}\), compact support of \(\psi\) gives
\[
\int_{\T^3}B_0(\xvec)\dd\xvec =\int_K\left[ X_1(\tau)\int_D\partial_{s_2}\psi(\boldsymbol{s})\dd\boldsymbol{s} -X_2(\tau)\int_D\partial_{s_1}\psi(\boldsymbol{s})\dd\boldsymbol{s} \right]\dd\tau =0.
\]
Therefore \(B_0\) is also mean-free.
\end{proof}

Proposition \ref{prop:bounded-section-Eu} readily implies the claim of Theorem \ref{thm:reversed-flow-decay}.

\begin{proof}[Proof of Theorem \ref{thm:reversed-flow-decay}]
Let \(B_0\) be the field constructed in Proposition~\ref{prop:bounded-section-Eu}.  Denote by \(\widetilde\phi_t\) the flow of \(\widetilde u_\alpha\).  On the first period it is given by
\[
\widetilde\phi_t =\phi_{1-t}\circ T^{-1}, \quad 0\leq t\leq1,
\]
and hence \(\widetilde\phi_1=T^{-1}\).  At integer times, the push-forward formula for the induction equation reads
\[
B(n,T^{-n}\xvec) =DT^{-n}(\xvec)B_0(\xvec).
\]
Since \(B_0(\xvec)\in E^u(\xvec)\) almost everywhere, \eqref{eq:inverse-contraction-on-Eu} gives
\[
|B(n,T^{-n}\xvec)| \leq\Lambda_u^{-n}|B_0(\xvec)|.
\]
The map \(T^{-n}\) preserves Lebesgue measure, so taking the \(L^p\) norm
proves the claim in Theorem \ref{thm:reversed-flow-decay} at integer times. For the estimate for all real times we set
\[
M_\alpha =\sup_{0\leq\tau\leq1} \esssup_{\xvec\in\T^3} \|D\widetilde\phi_\tau(\xvec)\|<\infty.
\]
If \(t=n+\tau\), with \(n\in\N\) and \(0\leq\tau<1\), we can write
\[
\| B(t,\cdot)\|_{L^p} \leq M_\alpha\| B(n,\cdot)\|_{L^p} \leq M_\alpha\Lambda_u^{-n}\|B_0\|_{L^p}.
\]
Since \(n\geq t-1\), there holds $\Lambda_u^{-n} \leq\Lambda_u \e^{-(\log\Lambda_u)t}$. Thus the claim of Theorem \ref{thm:reversed-flow-decay} follows with constants $\lambda_2 = \log\Lambda_u$ and \(C=M_\alpha\Lambda_u\).
\end{proof}

This result does not contradict the triviality of the stable kernel proved in Section \ref{s:L2-kernel}.  The stable bundle for the reversed map is the original two-dimensional bundle \(E^u\), not the one-dimensional bundle \(E^s\). Moreover, the bunching inequality from Lemma \ref{lem:dual-PDE-exclusion} is not available for \(T^{-1}\) since $\Lip(T^{-1})\asymp \alpha^3$, see \eqref{eq:inverse-DT}, and $\Lambda_u\asymp\alpha$.

\appendix

\section{Exponential growth with a hyperbolic point}\label{s:hyperbolic-point}

In this section we will show that any divergence-free autonomous smooth vector field with a hyperbolic point defines a non-universal ideal dynamo. In particular we obtain exponential stretching in $L^p$ for any $p>1$. The limit case $p=1$ is not covered here but only linear growth is to be expected, see \cite{BrueCotiZelatiMarconi24}. We will argue in $\T^2$ for simplicity. The same argument extends to three-dimensional hyperbolic equilibria with a one-dimensional unstable space. For a two-dimensional unstable space, the admissible range of $p$ depends on the unstable exponents.

\begin{lemma}\label{lemma:A1}
Let $u\in C^\infty(\T^2)$ be a divergence-free velocity field with flow map $\phi_t$. Suppose that there exist $\xvec_\ast\in\T^2$ and a positive number $\mu>0$ such that
\[
u(\xvec_\ast)=0, \quad \mathrm{spec}\, Du(\xvec_\ast) = \{\mu,-\mu\}.
\]
Fix $p>1$, then there exist a divergence-free field $B_0\in C^\infty(\T^2)$ and constants $\lambda,C>0$ such that
\[
\int_{\T^2} |D\phi_t(\xvec)B_0(\xvec)|^p\dd\xvec \geq C\e^{\lambda t}, \quad \text{for all }t\geq 0.
\]
\end{lemma}

Lemma \ref{lemma:A1} covers the claim for $1<p<\infty$. The endpoint case $p=\infty$ also holds true but follows from the pointwise estimate.

\begin{proof}
Let $U\subset\T^2$ be small rectangle around the hyperbolic point $\xvec_\ast$, and let us denote by $e_u$ the unstable direction for the hyperbolic point. The contraction principle and the local stable manifold theorem for the ODE give a local stable \(C^1\) curve tangent at \(\xvec_*\) to the one-dimensional stable eigenspace of \(Du(\xvec_*)\), whose forward trajectories converge exponentially to \(\xvec_*\), see e.g.\ \cite{Teschl12}*{Chapter 9.2}. Choose a nontrivial compact segment \(\Gamma\) of this curve whose forward orbit remains in some \(U_0\) compactly contained in $U$. Fix $\eps>0$, let $r>0$ and define the set
\begin{equation}\label{eq:Et}
E_t=\left\{\xvec\in\T^2 :\operatorname{dist}(\xvec,\Gamma) <r \e^{-(\mu+\varepsilon)t}\right\}.
\end{equation}
Then, we claim that there exist $c_0,c_1>0$ such that, on the one hand
\begin{equation}\label{eq:app1}
|E_t|\geq c_0e^{-(\mu+\varepsilon)t} \quad \text{and} \quad \phi_s(E_t)\subset U \quad \text{for all } 0\leq s\leq t,
\end{equation}
and on the other hand,
\begin{equation}\label{eq:app2}
|D\phi_t(\xvec)e_u| \geq c_1e^{(\mu-\varepsilon)t} \quad\text{for every }\xvec\in E_t.
\end{equation}
Once this is settled, choose a fixed compact tube $K\subset\interior U$ that contains $E_t$ for all $t\geq 0$. Take a smooth function $\psi:\T^2\to\R$ supported in $U$ such that
\[
B_0=\nabla^\perp\psi = e_u \quad \text{on } K.
\]
Such a function can be obtained by multiplying an affine function \(\ell\) satisfying \(\nabla^\perp\ell=e_u\) by a cut-off that equals one on \(K\). In particular, $B_0$ is smooth and divergence-free by construction. Therefore, using first $\eqref{eq:app2}$ and after \eqref{eq:app1}, we can write
\[
\int_{\T^2}|D\phi_t(\xvec)B_0(\xvec)|^p\dd \xvec \geq \int_{E_t}|D\phi_t(\xvec)e_u|^p\dd\xvec \geq c_1^p \e^{p(\mu-\eps)t} |E_t| \geq c_0c_1^p \e^{((p-1)\mu - (p+1)\eps)t}.
\]
Given $\mu>0$, for any $p>1$ there exists $\eps>0$ such that
\[
\lambda= (p-1)\mu - (p+1)\eps >0,
\]
and the claim of the lemma follows.

For completeness, we will also give a proof for the estimates \eqref{eq:app1} and \eqref{eq:app2}. After translating \(\xvec_*\in\T^2\) to the origin and making a linear change of coordinates that sends the eigenbasis to $e_u\mapsto e_1=(1,0)$, $e_s\mapsto (0,1)$, we have the matrix
\[
Du(0)=
\begin{pmatrix}
\mu&0\\
0&-\mu
\end{pmatrix}.
\]
Consider the difference between two trajectories $z(s) = \phi_s(\xvec)-\phi_s(\yvec)$. Via the fundamental theorem of calculus, $z$ solves the equation
\[
\dot z(s) = (Du(0) + \overline{R}(s))z(s), \quad z(0) = \xvec-\yvec,
\]
with
\[
\overline{R}(s) = \int_0^1 \left[Du(\phi_s(\yvec) + \theta z(s)) - Du(0)\right] \dd\theta.
\]
Moreover, if the segment joining $\phi_s(\xvec)$ and $\phi_s(\yvec)$ lies in $U$, upon shrinking $U$, there holds that $\|\overline{R}(s)\|<\delta$ for any arbitrarily small $\delta>0$. In such case we can compute
\[
\frac{1}{2}\frac{\dd}{\dd s}|z(s)|^2 = z(s)\cdot Du(0)z(s) + z(s)\cdot\overline{R}(s)z(s) \leq \mu|z(s)|^2 + \delta|z(s)|^2.
\]
For any $\delta<\eps$, via Grönwall we get
\begin{equation}\label{eq:app3}
|\phi_s(\xvec)-\phi_s(\yvec)| \leq \e^{(\mu+\eps)s}|\xvec-\yvec|.
\end{equation}
Choose a compact segment $\Gamma$ of the local stable curve such that $\phi_s(\Gamma)\subset U_0$ for some open set $U_0\subset\interior U$. Let $d = \dist(\overline{U_0},\partial U)>0$ and choose $r\in (0,d/2)$ in the definition of $E_t$ in \eqref{eq:Et}. Given $\xvec\in E_t$, choose $\yvec\in\Gamma$ such that $|\xvec-\yvec|<r\e^{-(\mu+\eps)t}$. 
Suppose now that there exists a first time \(\tau\in[0,t]\) such that
\[
|\phi_\tau(\xvec)-\phi_\tau(\yvec)|=\frac d2.
\]
For every \(0\leq s\leq\tau\), the point \(\phi_s(\yvec)\) belongs to \(U_0\) and the segment joining \(\phi_s(\xvec)\) to \(\phi_s(\yvec)\) lies in \(U\). Hence \eqref{eq:app3} applies on \([0,\tau]\), and
\[
|\phi_s(\xvec)-\phi_s(\yvec)| \leq \e^{(\mu+\varepsilon)s}|\xvec-\yvec| <r\e^{-(\mu+\varepsilon)(t-s)} \leq r<\frac d2
\]
for every \(0\leq s\leq\tau\). Taking \(s=\tau\) contradicts the definition of \(\tau\). Therefore no such exit time exists, and $\phi_s(\xvec)\in U$ for every $0\leq s\leq t$.
Moreover, since $\Gamma$ is a $C^1$ curve, the area of a sufficiently thin tube of width $r\e^{-(\mu+\eps)t}$ around $\Gamma$ is bounded below by the width of the tube times a multiplicative factor. This yields \eqref{eq:app1}.

To prove \eqref{eq:app2} we define the vector $w(s)=D\phi_s(\xvec)e_1$ along an orbit $\phi_s$ contained in $U$. Then $w$ solves the equation
\[
\dot w(s) = (Du(0) + R(s))w, \quad w(0)=e_1,
\]
with $R(s) = Du(\phi_s)-Du(0)$. As before, $\|R(s)\|<\delta$ as long as $\phi_s\in U$. Fix $\kappa<1$ and write $w = (w_1,w_2)$. We claim that the cone $\{w:|w_2|\leq \kappa|w_1|\}$ is forward invariant. Indeed we have
\[
\begin{array}{rcl}
    \dot w_1 & = & (\mu+R_{11})w_1 + R_{12}w_2,  \\
    \dot w_2 & = & R_{21}w_1 + (R_{22}-\mu)w_2.
\end{array}
\]
As long as $w_1\neq 0$ we can define the quotient $q(s) =w_2(s)/w_1(s)$, that satisfies the equation
\[
\dot q = R_{21} + (R_{22}-R_{11}-2\mu)q - R_{12}q^2,
\]
where we recall that $|R_{ij}(s)|<\delta$. At the boundaries of the cone $q\in\{-\kappa,\kappa\}$ we get
\[
\begin{array}{rcll}
    \dot q & \leq & -2\mu\kappa + \delta(1+\kappa)^2 & \text{if } q=\kappa,  \\
    \dot q & \geq & 2\mu\kappa - \delta(1+\kappa)^2 & \text{if } q=-\kappa.
\end{array}
\]
In particular, choosing $\delta>0$ small enough so that $2\mu\kappa > \delta(1+\kappa)^2$ yields that the vector field points inwards at both boundaries, hence by continuity we conclude that the (unstable) cone $\{w:|w_2|\leq \kappa|w_1|\}$ is forward invariant. To make sure that $w_1>0$ we can simply notice that $w_1(0)=1$, so by continuity we apply the argument on the maximal interval on which both $w_1>0$ and $|q|\leq w$. Moreover, inside this cone we have,
\[
\dot w_1 \geq (\mu-|R_{11}|-|R_{12}||q|)w_1 \geq (\mu - \delta(1+\kappa))w_1, \quad w_1(0) = 1.
\]
Another application of Grönwall's inequality gives $w_1(s)\geq \e^{(\mu-\eps)s}$, provided that we choose $\delta>0$ small enough so that $\delta(1+\kappa)\leq \eps$. Upon undoing the previous change of coordinates $(1,0)\mapsto e_u$, $(0,1)\mapsto e_s$, this estimate implies \eqref{eq:app2} and the proof concludes.
\end{proof}

\addtocontents{toc}{\protect\setcounter{tocdepth}{0}}

\section*{Acknowledgement}

The author thanks David Villringer for the enlightening discussions and insightful comments about Section \ref{s:reversed-flow}. The author also thanks Michele Coti Zelati and Massimo Sorella for the very valuable feedback. The author gratefully acknowledges support by the ERC/EPSRC Horizon Europe Guarantee EP/X020886/1.

\addtocontents{toc}{\protect\setcounter{tocdepth}{1}}

\bibliographystyle{abbrv}
\bibliography{dynamo.bib}

\begin{bibdiv}
\begin{biblist}

\bib{ACM19}{article}{
      author={Alberti, Giovanni},
      author={Crippa, Gianluca},
      author={Mazzucato, Anna~L.},
       title={Exponential self-similar mixing by incompressible flows},
        date={2019},
        ISSN={0894-0347,1088-6834},
     journal={J. Amer. Math. Soc.},
      volume={32},
      number={2},
       pages={445\ndash 490},
         url={https://doi.org/10.1090/jams/913},
      review={\MR{3904158}},
}

\bib{ArnoldsProblems}{book}{
      author={Arnold, V.I.},
       title={Arnold's problems},
   publisher={Springer Berlin Heidelberg},
        date={2004},
        ISBN={9783540207481},
         url={https://books.google.co.uk/books?id=vjTnwwLeQZsC},
}

\bib{AK98}{book}{
      author={Arnold, Vladimir~I.},
      author={Khesin, Boris~A.},
       title={Topological methods in hydrodynamics},
     edition={Second},
      series={Applied Mathematical Sciences},
   publisher={Springer, Cham},
        date={[2021] \copyright 2021},
      volume={125},
        ISBN={978-3-030-74277-5; 978-3-030-74278-2},
         url={https://doi.org/10.1007/978-3-030-74278-2},
      review={\MR{4268535}},
}

\bib{Aubin82}{book}{
      author={Aubin, Thierry},
       title={Nonlinear analysis on manifolds. {M}onge-{A}mp\`ere equations},
      series={Grundlehren der mathematischen Wissenschaften [Fundamental
  Principles of Mathematical Sciences]},
   publisher={Springer-Verlag, New York},
        date={1982},
      volume={252},
        ISBN={0-387-90704-1},
         url={https://doi.org/10.1007/978-1-4612-5734-9},
      review={\MR{681859}},
}

\bib{BaxendaleRozovskii93}{article}{
      author={Baxendale, P.~H.},
      author={Rozovskii, B.~L.},
       title={Kinematic dynamo and intermittence in a turbulent flow},
        date={1993},
     journal={Geophysical \& Astrophysical Fluid Dynamics},
      volume={73},
      number={1-4},
       pages={33\ndash 60},
}

\bib{BBPS22}{article}{
      author={Bedrossian, Jacob},
      author={Blumenthal, Alex},
      author={Punshon-Smith, Samuel},
       title={Almost-sure exponential mixing of passive scalars by the
  stochastic {N}avier-{S}tokes equations},
        date={2022},
        ISSN={0091-1798,2168-894X},
     journal={Ann. Probab.},
      volume={50},
      number={1},
       pages={241\ndash 303},
         url={https://doi.org/10.1214/21-aop1533},
      review={\MR{4385127}},
}

\bib{BCZG23}{article}{
      author={Blumenthal, Alex},
      author={Coti~Zelati, Michele},
      author={Gvalani, Rishabh~S.},
       title={Exponential mixing for random dynamical systems and an example of
  {P}ierrehumbert},
        date={2023},
        ISSN={0091-1798,2168-894X},
     journal={Ann. Probab.},
      volume={51},
      number={4},
       pages={1559\ndash 1601},
         url={https://doi.org/10.1214/23-aop1627},
      review={\MR{4597327}},
}

\bib{BrueCotiZelatiMarconi24}{article}{
      author={Bru\`e, Elia},
      author={{Coti Zelati}, Michele},
      author={Marconi, Elio},
       title={Enhanced dissipation for two-dimensional {H}amiltonian flows},
        date={2024},
        ISSN={0003-9527,1432-0673},
     journal={Arch. Ration. Mech. Anal.},
      volume={248},
      number={5},
       pages={Paper No. 84, 37},
         url={https://doi.org/10.1007/s00205-024-02034-3},
      review={\MR{4797689}},
}

\bib{ChildressGilbert}{book}{
      author={Childress, Stephen},
      author={Gilbert, Andrew~D.},
       title={Stretch, twist, fold: The fast dynamo},
     edition={First},
      series={Monographs},
   publisher={Springer Berlin, Heidelberg},
        date={2008},
      volume={37},
        ISBN={978-3-662-14014-7},
         url={https://doi.org/10.1007/978-3-540-44778-8},
}

\bib{CoopermanRowan26}{article}{
      author={Cooperman, William},
      author={Rowan, Keefer},
       title={Exponential scalar mixing for the 2{D} {N}avier-{S}tokes
  equations with degenerate stochastic forcing},
        date={2026},
        ISSN={0020-9910,1432-1297},
     journal={Invent. Math.},
      volume={243},
      number={3},
       pages={863\ndash 959},
         url={https://doi.org/10.1007/s00222-025-01384-3},
      review={\MR{5008155}},
}

\bib{CotiZelatiNavarroFernandez}{article}{
      author={Coti~Zelati, Michele},
      author={Navarro-Fern\'andez, V\'ictor},
       title={Three-dimensional exponential mixing and ideal kinematic dynamo
  with randomized {ABC} flows},
        date={2026},
     journal={J. Dyn. Diff. Equat.},
}

\bib{CZSV25}{unpublished}{
      author={Coti~Zelati, Michele},
      author={Sorella, Massimo},
      author={Villringer, David},
       title={Alpha-unstable flows and the fast dynamo problem},
        date={2025},
         url={https://arxiv.org/abs/2504.00855},
        note={Preprint arXiv:2504.00855},
}

\bib{CZVS26+}{misc}{
      author={Coti~Zelati, Michele},
      author={Sorella, Massimo},
      author={Villringer, David},
       title={A fast dynamo on the three-torus},
        date={2026},
         url={https://arxiv.org/abs/2603.09861},
        note={Preprint arXiv:2603.09861},
}

\bib{Cowling33}{article}{
      author={Cowling, T.~G.},
       title={The magnetic field of sunspots},
        date={1933},
        ISSN={0035-8711},
     journal={Mon. Not. R. Astron. Soc.},
      volume={94},
       pages={39\ndash 48},
}

\bib{Davidson01}{book}{
      author={Davidson, P.~A.},
       title={An introduction to magnetohydrodynamics},
      series={Cambridge Texts in Applied Mathematics},
   publisher={Cambridge University Press, Cambridge},
        date={2001},
        ISBN={0-521-79487-0},
         url={https://doi.org/10.1017/CBO9780511626333},
      review={\MR{1825486}},
}

\bib{DNFLM26+}{misc}{
      author={Del~Nin, Giacomo},
      author={Faraco, Daniel},
      author={Lindberg, Sauli},
      author={Mengual, Francisco},
       title={Turbulent dynamos on bounded domains and their generalization to
  the geometric transport equation},
        date={2026},
         url={https://arxiv.org/abs/2605.20451},
        note={Preprint arXiv:2605.20451},
}

\bib{DemersLiverani08}{article}{
      author={Demers, Mark~F.},
      author={Liverani, Carlangelo},
       title={Stability of statistical properties in two-dimensional piecewise
  hyperbolic maps},
        date={2008},
        ISSN={0002-9947,1088-6850},
     journal={Trans. Amer. Math. Soc.},
      volume={360},
      number={9},
       pages={4777\ndash 4814},
         url={https://doi.org/10.1090/S0002-9947-08-04464-4},
      review={\MR{2403704}},
}

\bib{ELM25}{article}{
      author={Elgindi, Tarek~M.},
      author={Liss, Kyle},
      author={Mattingly, Jonathan~C.},
       title={Optimal enhanced dissipation and mixing for a time-periodic,
  {L}ipschitz velocity field on {$\mathbb{T}^2$}},
        date={2025},
        ISSN={0012-7094,1547-7398},
     journal={Duke Math. J.},
      volume={174},
      number={7},
       pages={1209\ndash 1260},
         url={https://doi.org/10.1215/00127094-2024-0057},
      review={\MR{4912975}},
}

\bib{ElgindiZlatos19}{article}{
      author={Elgindi, Tarek~M.},
      author={Zlato\v{s}, Andrej},
       title={Universal mixers in all dimensions},
        date={2019},
        ISSN={0001-8708,1090-2082},
     journal={Adv. Math.},
      volume={356},
       pages={106807, 33},
         url={https://doi.org/10.1016/j.aim.2019.106807},
      review={\MR{4008523}},
}

\bib{Freedman99}{incollection}{
      author={Freedman, Michael~H.},
       title={Zeldovich's neutron star and the prediction of magnetic froth},
        date={1999},
   booktitle={The {A}rnoldfest ({T}oronto, {ON}, 1997)},
      series={Fields Inst. Commun.},
      volume={24},
   publisher={Amer. Math. Soc., Providence, RI},
       pages={165\ndash 172},
         url={https://doi.org/10.1090/fic/024/11},
      review={\MR{1733574}},
}

\bib{gilbert1988}{article}{
      author={Gilbert, Andrew~D.},
       title={Fast dynamo action in the {P}onomarenko dynamo},
        date={1988},
     journal={Geophysical \& Astrophysical Fluid Dynamics},
      volume={44},
      number={1-4},
       pages={241\ndash 258},
}

\bib{GilbertVanneste21}{article}{
      author={Gilbert, Andrew~D.},
      author={Vanneste, Jacques},
       title={A geometric look at {MHD} and the {B}raginsky dynamo},
        date={2021},
        ISSN={0309-1929,1029-0419},
     journal={Geophys. Astrophys. Fluid Dyn.},
      volume={115},
      number={4},
       pages={436\ndash 471},
         url={https://doi.org/10.1080/03091929.2020.1839896},
      review={\MR{4301417}},
}

\bib{KatokHasselblatt95}{book}{
      author={Katok, Anatole},
      author={Hasselblatt, Boris},
       title={Introduction to the modern theory of dynamical systems},
      series={Encyclopedia of Mathematics and its Applications},
   publisher={Cambridge University Press, Cambridge},
        date={1995},
      volume={54},
        ISBN={0-521-34187-6},
         url={https://doi.org/10.1017/CBO9780511809187},
        note={With a supplementary chapter by Katok and Leonardo Mendoza},
      review={\MR{1326374}},
}

\bib{Larmor19}{article}{
      author={Larmor, J.},
       title={Possible rotational origin of magnetic fields of {S}un and
  {E}arth},
        date={1919},
     journal={Electrical Review},
      volume={85},
      number={512},
}

\bib{LissMattingly26+}{misc}{
      author={Liss, Kyle~L.},
      author={Mattingly, Jonathan~C.},
       title={The {B}atchelor spectrum for a deterministically driven passive
  scalar},
        date={2026},
         url={https://arxiv.org/abs/2603.08904},
        note={Preprint arXiv:2603.08904},
}

\bib{MarsdenHughes94}{book}{
      author={Marsden, Jerrold~E.},
      author={Hughes, Thomas J.~R.},
       title={Mathematical foundations of elasticity},
   publisher={Dover Publications, Inc., New York},
        date={1994},
        ISBN={0-486-67865-2},
        note={Corrected reprint of the 1983 original},
      review={\MR{1262126}},
}

\bib{Moffatt1978}{book}{
      author={Moffatt, H.~K.},
       title={Magnetic field generation in electrically conducting fluids},
   publisher={Cambridge University Press},
        date={1978},
}

\bib{NFSeis26}{article}{
      author={Navarro-Fern\'andez, V\'ictor},
      author={Seis, Christian},
       title={Exponential mixing by random cellular flows},
        date={2026},
        ISSN={0022-1236,1096-0783},
     journal={J. Funct. Anal.},
      volume={290},
      number={2},
       pages={Paper No. 111227, 54},
         url={https://doi.org/10.1016/j.jfa.2025.111227},
      review={\MR{4973604}},
}

\bib{NFVillringer}{misc}{
      author={Navarro-Fern\'andez, V\'ictor},
      author={Villringer, David},
       title={Spectral instability in the smooth {P}onomarenko dynamo},
        date={2025},
        note={Preprint arXiv:2509.19201},
}

\bib{NFV26}{article}{
      author={Navarro-Fern\'andez, V\'ictor},
      author={Villringer, David},
       title={Nonlinear instability for the 3{D} {MHD} equations around the
  {T}aylor-{C}ouette flow},
        date={2026},
        ISSN={0951-7715,1361-6544},
     journal={Nonlinearity},
      volume={39},
      number={5},
       pages={Paper No. 055009, 11},
         url={https://doi.org/10.1088/1361-6544/ae65fc},
      review={\MR{5073797}},
}

\bib{Niebel6+}{misc}{
      author={Niebel, Lukas},
       title={An autonomous {L}ipschitz fast dynamo on the three-torus},
        date={2026},
         url={https://arxiv.org/abs/2608.02586},
        note={Preprint arXiv:2608.02586},
}

\bib{Rowan25}{misc}{
      author={Rowan, Keefer},
       title={A subsequentially fast dynamo on $\mathbb{T}^3$},
        date={2025},
         url={https://arxiv.org/abs/2505.23936},
        note={Preprint arXiv:2505.23936},
}

\bib{SorellaVillringer25+}{misc}{
      author={Sorella, Massimo},
      author={Villringer, David},
       title={A limsup fast dynamo on {$\T^3$}},
        date={2025},
         url={https://arxiv.org/abs/2511.23024},
        note={Preprint arXiv:2511.23024},
}

\bib{Teschl12}{book}{
      author={Teschl, Gerald},
       title={Ordinary differential equations and dynamical systems},
      series={Graduate Studies in Mathematics},
   publisher={American Mathematical Society, Providence, RI},
        date={2012},
      volume={140},
        ISBN={978-0-8218-8328-0},
         url={https://doi.org/10.1090/gsm/140},
      review={\MR{2961944}},
}

\bib{david-thesis}{thesis}{
      author={Villringer, David},
       title={On spectral instability and kinematic dynamo action in
  three-dimensional incompressible flows},
        type={Ph.D. Thesis},
     address={Imperial College London, UK},
        date={2026},
}

\bib{Vishik89}{article}{
      author={Vishik, M.~M.},
       title={Magnetic field generation by the motion of a highly conducting
  fluid},
        date={1989},
        ISSN={0309-1929,1029-0419},
     journal={Geophys. Astrophys. Fluid Dynam.},
      volume={48},
      number={1-3},
       pages={151\ndash 167},
         url={https://doi.org/10.1080/03091928908219531},
      review={\MR{1024693}},
}

\bib{ZRMS1984}{article}{
      author={Zeldovich, Ya.~B.},
      author={Ruzmaikin, A.~A.},
      author={Molchanov, S.~A.},
      author={Sokoloff, D.~D.},
       title={Kinematic dynamo problem in a linear velocity field},
        date={1984},
     journal={Journal of Fluid Mechanics},
      volume={144},
       pages={1\ndash 11},
}

\bib{zeldovich1980magnetic}{article}{
      author={Zeldovich, Yakov~Borisovich},
      author={Ruzmaikin, A.A.},
       title={Magnetic field of a conducting fluid in two-dimensional motion},
        date={1980},
     journal={Zhurnal Eksperimental'noi i Teoreticheskoi Fiziki},
      volume={78},
       pages={980\ndash 986},
}

\end{biblist}
\end{bibdiv}

\end{document}